\documentclass[11pt,reqno]{article}
\usepackage{amsgen, amsmath, amsfonts, amsthm, amssymb,enumerate,amscd,graphics,dsfont,comment,tikz-cd,tikz,hyperref,subcaption,float}

\definecolor{blue}{rgb}{0.0, 0.43, 0.54}
\definecolor{purple}{rgb}{0.45,0.31,0.59}
\hypersetup{
    colorlinks,
    linkcolor={blue},
    citecolor={purple},
    urlcolor={blue!80!black}
}

\usetikzlibrary{intersections}

\newtheorem{defn}{Definition}[section]
\newtheorem{prop}[defn]{Proposition}
\newtheorem{lem}[defn]{Lemma}
\newtheorem{thm}[defn]{Theorem}
\newtheorem{cor}[defn]{Corollary}
\newtheorem{rem}[defn]{Remark}

\newcommand {\ZZ}{{\mathds Z}}

\newcommand {\XX}{{\mathcal X}}

\newcommand {\C}{{\mathds C}}

\newcommand {\Z}{{\mathbb Z}}

\newcommand {\Q}{{\mathds Q}}

\newcommand {\OO}{{\mathcal O}}

\newcommand {\UH}{{\mathcal H}}

\newcommand {\TV}{{\tilde{V}}}

\newcommand {\M}{{\mathcal M}}

\newcommand {\CP}{{\mathds P}}

\newcommand {\ut}{{\underline{t} }}

\newcommand {\TK}{{\tilde{K}}}

\newcommand {\TD}{{\tilde{D}}}

\newcommand {\TQ}{{\tilde{Q}}}

\def\rank{\operatorname{\rank}}

\def\End{\operatorname{End}}

\def\dim{\operatorname{dim}}

\def\res{\operatorname{Res}}
\def\disc{\operatorname{disc}}

\title{Equations defining Jacobians with Real Multiplication}
\author{Rahul Mistry and Ramesh Sreekantan}

\begin{document}
	\baselineskip=17pt
	
	\maketitle

\tableofcontents

\begin{abstract}
	
	If $C:y^2=x(x-1)(x-a_1)(x-a_2)(x-a_3)$ is genus $2$ curve a natural question to ask is: Under what conditions on $a_1,a_2,a_3$ does the Jacobian $J(C)$ have real multiplication by $\ZZ[\sqrt{\Delta}]$ for some $\Delta>0$. Over a hundred years ago Humbert \cite{humb} gave an answer to this question for $\Delta=5$ and $\Delta=8$. In this paper we use work of Birkenhake and Wilhelm \cite{biwi} along with some classical results in enumerative geometry to generalize this to all discriminants, in principle. We also work it out  explicitly in a few more cases. 

\end{abstract}

\section{Introduction}

\subsection{Curves of genus $2$ and Abelian surfaces}

Let 
$$C=C_{a_1,a_2,a_3}:y^2=x(x-1)(x-a_1)(x-a_2)(x-a_3)$$
$a_1, a_2, a_3 \in \CP^1(\C)\backslash \{0,1,\infty\}$ be a genus $2$ (hyperelliptic) curve over $\C$. The Jacobian $A=J(C)$ is a principally polarized  Abelian surface. A natural question to ask is -`What is the endomorphism ring of $A$?'.

$A$ determines a point in moduli space - the Siegel modular threefold with full level $2$ structure, $S_2(2)$, as the points $a_1,a_2,a_3$ and $0,1,\infty$ determine the level two structure. Most Abelian surfaces have endomorphism ring $\ZZ$ but one has submoduli corresponding to those Abelian surfaces with larger endomorphism rings. So the question is translated to `When does the point corresponding to $A$ lie on a special submodulus?'

\subsection{Endomorphism rings of Abelian surfaces}

There are four possibilities for the endomorphism ring of an Abelian surface over $\C$.

\begin{itemize}
	
	\item $\End(A) \supseteq \ZZ$ and generically one has equality. 
	 
	\item $\End(A) \supseteq \ZZ[\Delta]$ where $\Delta>0$ - this happens on surfaces which are called Humbert surfaces $H_{\Delta}$. The moduli of products of elliptic curves corresponds to the case $\Delta=1$. Here too, generically one has equality. 
	
	\item $\End(A) \supseteq\OO$ where $\OO$ is an order in a indefinite quaternionic algebra over $\Q$  - either a division algebra  $B$ or a subring of $M_2(\Q)$. This happens on curves called Shimura curves, in the first case, or modular curves, in the second. In this case as well generically one has equality. 
	
	\item $\End(A)=\OO_D$, where $\OO_D$ is a subring of $M_2(K)$ where $K=\Q(\sqrt{D})$ is an imaginary quadratic field. In this case $A$ is isogenous to the self product of an elliptic curve with complex multiplication by $\Q(\sqrt{D})$. 
	
\end{itemize}

\subsection{The endomorphism ring and the Neron-Severi group}
The endomorphism ring of an Abelian variety is closely connected with its Neron-Severi group. There is an involution on the endomorphism ring of the Abelian variety determined by the principal polarization called the Rosati involution. The Neron-Severi group can be identified with the elements of the endomorphism ring fixed under this involution.

In the opposite direction, let $L_0$ denote the class of the polarization and 
 $PNS(A)=L_0^{\perp}$ denote its  orthogonal complement in the Neron-Severi group. This is usually called the {\em Primitive Neron-Severi} group.

Recall that the {\em Clifford Algebra} of a lattice $\Lambda$ with a quadratic form $Q$ is defined to be 
$$CL(\Lambda,Q)=T(\Lambda)/<x\otimes x -Q(x)1>$$
and the dimension of $CL(\Lambda,Q)=2^{\dim_{\Q}(\Lambda \otimes \Q)}$. Here $T(\Lambda)$ is the Tensor algebra 
$$T(\Lambda)=\bigoplus_{k=0}^{\infty} \bigotimes^k \Lambda=\ZZ \bigoplus \Lambda \bigoplus \Lambda \otimes \Lambda \bigoplus \cdots$$
From the Hodge Index Theorem the intersection pairing is negative definite on $PNS(A)$. Let 
$<,>=-2(,)$. Then $<,>$ is a positive definite bilinear form on $PNS(A)$ called the {\em Humbert Pairing}. For $D$ in $NS(A)$, 
$$<D,D>=(D,L_0)^2-2(D,D)$$ 
Let $H$ denote the corresponding quadratic form. 

The endomorphism ring is then the  Clifford algebra of the primitive Neron-Severi with respect to the negative of the intersection pairing. 
$$\End(A)=Cl(PNS(A),H)$$ 

For instance if the Abelian surface $A$ has real multiplication the rank of the Neron-Severi, usually called the Picard number  $\rho(A) \geq 2$. If $\rho(A)=2$, as it is generically, and $v$ is a generator of $PNS(A)$,
$$\End(A)=CL(\ZZ v,Q)=\ZZ[\sqrt{H(v)}].$$
$H(v)$ is usually called the {\em Humbert Invariant} and the moduli of Abelian surfaces where there is an element in $NS(A)$ of Humbert invariant $\Delta$ are called Humbert surfaces $H_{\Delta}$.

When $\rho(A)=3$, then $\End(A)$ is an indefinite quaternion algebra and when $\rho(A)=4$, $\End(A)$ is and order in $M_2(K)$ where $K$ is a CM field.

Hence the question becomes - under what conditions on $a_1,a_2$ and $a_3$ is $\rho(J(C))>1?$. Humbert gave a criterion under which the rank of the Neron-Severi is $2$ and the corresponding endomorphism algebra is $\ZZ[\sqrt{5}]$. He then used it to answer the question in the case when $\Delta=5$. 

Hashimoto and Murabayashi \cite{hamu} gave a modern exposition of Humbert's theorem and used that to obtain equations of some Shimura curves. In this article we generalize there work. There are two new ingredients - the first is a theorem of Birkenhake-Wilhelm generalizing Humbert's criteria for $\ZZ[\sqrt{5}]$ and $\ZZ[\sqrt{8}]$ and the second is an observation that a curious coincidence of enumerative geometry allows us to find certain relevant curves. 

Birkenhake and Wilhelm \cite{biwi} generalized Humbert's criterion to infinitely many $\Delta$. 
The aim of this paper is to show that their work can be used to obtain conditions on $a_1,a_2$ and $a_3$ as well along the lines of Humbert. In principle this can be done in an infinite number of cases but in practice it may be a little difficult. 

\section{Humbert's Theorem}

In order to state Humbert's theorem we need to introduce some auxiliary objects. 

\subsection{The Kummer surface}

If $(A,L_0)$ is a principally polarized Abelian surface, the line bundle $L_0^2$ defines a map 
$$\phi=\phi_{L_0^2}:A \rightarrow \CP^3$$
This is a double cover onto its image branched at the $16$ two torsion points of $A$. The image is called the {\em Kummer surface} and we denote it by $K_A$. This is a singular surface with nodes at the images if the $16$ two torsion points. 

The minimal desingularization $\TK_A$ of $K_A$ is a $K3$ surface and we call it  the {\em Kummer K3 surface}. In the literature sometimes this is referred to as the Kummer surface but as we have occasion to use both we make this distinction.

\subsection{The Kummer Plane}

Projecting from $\phi(0)$ defines a rational map $\pi: K_A \dashrightarrow \CP^2$ which can be extended to a morphism from the blow-up of $K_A$ at $\phi(0)$ to $\CP^2$. It is known \cite{biwi} that this is a double cover branched at 6 lines $l_i$ tangent to a conic. These lines meet at 15 points $q_{ij}=l_i \cap l_j$. The points $q_{ij}$ are the image of the $15$ non-zero two torsion points of $A$.  We call the configuration of $(\CP^2,l_1,\dots,l_6)$ the {\em Kummer Plane of $A$} and denote it by $\CP^2_A$.

\begin{center}
	\tikzset{every picture/.style={line width=0.75pt}} 
	
	\begin{tikzpicture}[x=0.75pt,y=0.75pt,yscale=-1,xscale=1]
		
		\draw    (94,115) -- (349,212.5) ;
		\draw    (141,73.5) -- (186,266.5) ;
		\draw    (282,73) -- (174,266.5) ;
		\draw    (98,237.5) -- (319,103.5) ;
		\draw    (81,185.5) -- (367,186.5) ;
		\draw    (124,84.5) -- (259,256.5) ;
		\draw   (53,56) -- (389,56) -- (389,273.5) -- (53,273.5) -- cycle ;
		
		\draw (77,112) node [anchor=north west][inner sep=0.75pt]   [align=left] {$\displaystyle l_{1}$};
		\draw (105,65) node [anchor=north west][inner sep=0.75pt]   [align=left] {$\displaystyle l_{2}$};
		\draw (151,71) node [anchor=north west][inner sep=0.75pt]   [align=left] {$\displaystyle l_{3}$};
		\draw (290,71) node [anchor=north west][inner sep=0.75pt]   [align=left] {$\displaystyle l_{4}$};
		\draw (313,113) node [anchor=north west][inner sep=0.75pt]   [align=left] {$\displaystyle l_{5}$};
		\draw (345,161) node [anchor=north west][inner sep=0.75pt]   [align=left] {$\displaystyle l_{6}$};
		\draw (270,189) node [anchor=north west][inner sep=0.75pt]    {$q_{16}$};
		\draw (160,244) node [anchor=north west][inner sep=0.75pt]    {$q_{34}$};
		\draw (136,139.4) node [anchor=north west][inner sep=0.75pt]    {$q_{13}$};
		\draw (360,66.4) node [anchor=north west][inner sep=0.75pt]    {$\CP^{2}$};
	\end{tikzpicture}
	
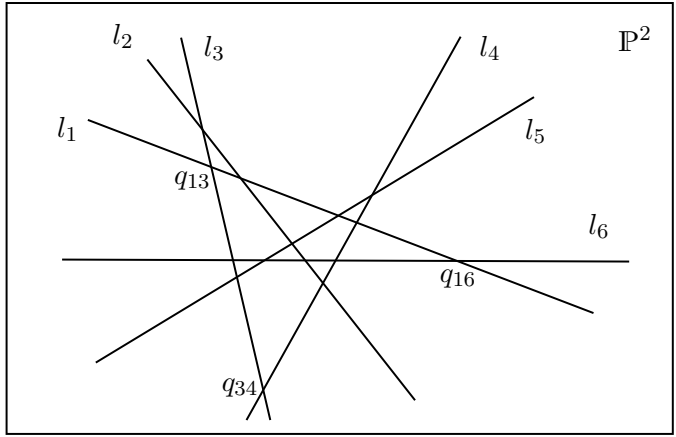
	{Heuristic representation of the Kummer Plane}
\end{center}

\subsection{Humbert's Theorem}

We can now state Humbert's theorem. 

\begin{thm}[Humbert]\cite{humb} Let $(A,L_0)$ be a principally polarized Abelian surface and let $\CP^2_A$ be its Kummer plane. Suppose there exists a conic $Q$ in $\CP^2_A$ passing through 5 of the 15 points $q_{ij}$ and meeting the remaining line tangentially, then the Abelian surface $A$ has real multiplication by $\ZZ[\sqrt{5}]$. Conversely, if $A$ has real multiplication by $\ZZ[\sqrt{5}]$ then  there exists such a conic in $\CP^2_A$. 
	\end{thm}

The class of the curve  $\phi^* \circ \pi^*(Q)$ in the Neron-Severi pulls back to a multiple of the class of the principle polarization $C$ hence is {\em not} a new element of the Neron-Severi. The reason this theorem holds is the following lemma. 

\begin{lem} Let $Q$ be a rational curve. If $f:D \rightarrow Q$ is a double cover such that the ramification points are singular. Then $D$ is reducible and has two components $D_1$ and $D_2$ such that $D_1$ and $D_2$ meet at the ramification points. 
\label{doublecover}
\end{lem}

\begin{proof}
Let $\tilde{f}:\TD \rightarrow \TQ$ be the map between the normalizations induced by the map $f$. This is a unramified double cover and $\TQ\simeq \CP^1$. However since $\CP^1$ is simply connected it has no unramified double covers. Hence it consists of two components $\TD_1$ and $\TD_2$. The images $D_1$ and $D_2$ under the map to $D$ are the two components of $D$ and the points of intersection are the ramification points. 
 
\end{proof}

One can then see that the pull back of one of the components $D_1$ or $D_2$  is a new element of the Neron-Severi group. 

In Humbert's case the conic $Q$ meets the ramified sextic $S$ only at double points. The double cover $D$ of $Q$ in $K_A$ has nodes at those points and hence satisfies the conditions of Lemma \ref{doublecover}. Therefore  $D$ is the union of two components $D_1$ and $D_2$.

Note that the involution $\sigma$ induced by the double cover satisfies $\sigma(D_1)=D_2$. Hence $D_1-D_2$ is orthogonal to the polarization class. Hence the primitive Neron-Severi contains $\ZZ[D_1-D_2]$ and the endomorphism ring contains $Cl(\ZZ[D_1-D_2],-2(,))$. Birkenhake-Willhelm \cite{biwi} (Proposition 6.4) compute this self intersection and from that one has that the Abelian surface has real multiplication by $\ZZ[\sqrt{5}]$. 

\subsection{Co-ordinates}

In order to apply Humbert's theorem to get a criteria for real multiplication we need to describe the points $q_{ij}$ and the lines $l_i$ in terms of the points $a_1, a_2$ and $a_3$. For this we use the setup of Hashimoto-Murabayashi \cite{hamu}. Let 
$$C:y^2=x(x-1)(x-a_1)(x-a_2)(x-a_3)$$
be a curve of genus $2$. Let $[x:y:z]$ denote the co-ordinates of $\CP^2$. One can choose the six lines to be 
$$l_i:y+2a_ix+a_i^2=0 \hspace{1in} 1 \leq i \leq 3$$
$$l_4:y+2x+1=0 \hspace{2cm} l_5:y=0 \hspace{2cm} l_6:z=0$$
From this one can compute the points $q_{ij}=l_i \cap l_j$
$$q_{ij}=[\frac{-(a_i+a_j)}{2}:a_ia_j:1] \text{ for } 1 \leq i,j \leq 3$$
and using $a_4=1, a_5=0$ and $a_6=\infty$ we get 
$$q_{i4}=[\frac{-(a_i+1)}{2}:a_1:1] \hspace{2cm}  q_{i5}=[-a_i:0:2] \hspace{2cm}  q_{i6}=[1:-2a_i:0]$$

We then have the following theorem of Humbert

\begin{thm}[Humbert] \cite{hamu}
	\label{conics}
	Let $(A,L_0)$ be a principally polarized Abelian surface  associated to the curve of genus $2$ given by \[y^2=x(x-1)(x-a_1)(x-a_2)(x-a_3)\text{ with }a_1,a_2,a_3\in\C-\{0,1\}\]
	and let $\End(A)$ denote the endomorphism ring of $(A,L_0)$. Then $\Q(\sqrt{5})\subseteq\End(A)\otimes \Q$ if
	\begin{align*}
		4(a_1^2a_3-a_2^2+a_3^2(1-a_1)+a_2-a_3)(a_1^2a_2a_3-a_1a_2^2a_3)\\
		=(a_1^2a_3(a_2+1)-a_2^2(a_1+a_3)+a_2a_3^2(1-a_1)+a_1(a_2-a_3))^2
	\end{align*}
	holds.
\end{thm}

\begin{proof} (Sketch). The idea is the following. There is a unique conic passing through any five points in $\CP^2$ in general position (that is, no three are co-linear). Consider the points $q_{12}, q_{23}, q_{34}, q_{45}$ and $q_{51}$. Let $Q=Q(x,y,z)$ be the conic passing through these five points. In general  $Q$ will meet the line $l_6: z=0$ at two points. These two points satisfy the quadratic equation $Q(x,y,0)$. The coefficients of $Q(x,y,0)$ are determined by $a_1,a_2,a_3$ and $a_4$. 
	
	The theorem of Humbert states that one has real multiplication by $\ZZ[\sqrt{5}]$ if $Q$ is tangent to $l_6$. This happens when the two roots coincide - or equivalently, when the discriminant of $Q(x,y,0)$ is $0$. Since the coefficients are determined by the $a_i$ the condition that the discriminant is $0$ is the expression above. 
	
\end{proof}
A different choice of points will result in a different equation. The moduli of Abelian surfaces with real multiplication by $\ZZ[\sqrt{5}]$ has $6$ components in the Siegel modular threefold of level $2$. These six components correspond to the different choices of lines $l_i$. 

There are other instances when a conic can meet the six lines everywhere with multiplicity $2$. For instance, there exists two conics passing through $4$ points in general position in $\CP^2$ and tangent to two lines and there exist four conics passing through $3$ points and tangent to three lines. Humbert showed that these correspond Abelian surfaces with Humbert invariant $4,8$ or $9$. 

\subsection{The Theorem of Birkenhake and Wilhelm}

Birkenhake and Wilhelm \cite{biwi} greatly generalized Humbert's theorem to give criteria for multiplication by $\ZZ[\sqrt{\Delta}]$ for an infinite number of $\Delta$. The idea is the same - if there is a {\em rational} curve of degree $d$ which meets the six lines at $3d$ points, then the normalization of the pre-image in $K_A$ is an unramified double cover of $\CP^1$ hence has two components. Birkenhake and Willhelm compute the Humbert invariant of one of the components to get their theorem. Conversely, for $\Delta$ of a particular type determined by the degree and points even multiplicity, they show that there exists such a rational curve. They have the following theorem. 

\begin{thm}[Birkenhake-Wilhelm] Let $d$  be a natural number and $3 \leq k \leq 12$ and let $\Delta$ be as below
	
		\vspace{\baselineskip}
	\begin{center}\label{valdelta}
		\begin{tabular}{||l|c|c|}
			\hline
			Case  & $\Delta$ & $d+k$  \\
			\hline
			$I.$ & $2d^2+7-2k$ & $d+k$ odd\\
			$II.$ & $2d^2+8-2k$ & $d+k$ even\\
			$III.$ & $(d+1)^2$ &$d+3$\\
			\hline
		\end{tabular}
	\end{center}
	\vspace{\baselineskip}
	
	Let $(A,L_0)$ be a principally polarized Abelian surface with a class $D_{\Delta}$ of Humbert invariant $\Delta$ and $S_A=\prod_{i=1}^6 l_ii$ be the (degenerate) sextic given by the product of ramified lines in the Kummer plane $\CP^2_A$. 
	
	Then if $\Delta$ is as above, there exists a rational curve $Q_{\Delta}$ on $\CP^2_A$ of degree $d$  which passes through $k$ of the points $q^{ij}$ with no singularities at those points and meets $S_A$ at the remaining points with even multiplicity. 
	
	Further, $Q_{\Delta} = \pi \circ  \phi(C_{\Delta} )$ where $C_{\Delta}$ is a curve on $A$ which lies in a class of the form $a\L_0 + bD_{\Delta}$ with $b\neq 0$. In particular, the class of $C_{\Delta}$ is not a multiple of the class of the principal polarization. 
	
	Conversely, if there is a rational curve of degree $d$ which passes through $k$ points $q_{ij}$ and meets the sextic at the  remaining points tangentially then there is a class of invariant $\Delta'$ for some $\Delta' \leq \Delta$, where $\Delta$ is as above.   
	
\label{birkwilh}
\end{thm}

Humbert's case above is the case when $d=2$ and $k=5$, so $d+k=7$ is odd. Then $2d^2+7-2k=8+7-10=5$. Along the lines of the expression in  Humbert's  Theorem \ref{conics} we have the following theorem. 

\begin{thm} Let $(A,L_0)$ be a principally polarized Abelian surface which is the Jacobian of a curve $C:y^2=x(x-1)(x-a_1)(x-a_2)(x-a_3)$. Then, for $\Delta$ as in  Theorem \ref{birkwilh} there is a quadratic form $B_{\Delta}$ with coefficients determined by the $a_i$ such that if the discriminant of the quadratic form is $0$ then $A$ has real multiplication by $\ZZ[\sqrt{\Delta'}]$ for some $\Delta' \leq \Delta$. 
\end{thm} 

\begin{proof} Recall that in  Humbert's we used the fact that there is a conic passing through any five points in general position. By a curious accident of enumerative geometry a similar statement holds in general. If $Q$ is a rational curve of degree $d$ then it will meet the sextic $S=\bigcup_{n=1}^6 l_i$ at $6d$ points. However, the theorem of Birkenhake and Willhelm states that if all the points of intersection are double points then the corresponding Abelian surface has real multiplication by some $\Delta$.  So such a rational curve determines $3d$ points - say $k$ points of the form $q_{ij}$ and $3d-k$ remaining points of even multiplicity (counted carefully). The idea is then to argue that for {\em any} $S$ there is a rational curve of degree $d$ which passes through $3d-1$ points. Enumerative geometry suggests that it is likely that such a statement holds - as there always exists a rational curve passing through $3d-1$ points in general position in $\CP^2$. More generally, there exists rational curves passing through $3d-k$ points and tangent to $k-1$ lines. 
	
	However, that does not immediately imply that it holds for $3d-1$ of the points of intersection of $S$ with a rational curve. For that we need the following theorem - the proof of which was suggested to us by Tom Graber.

\begin{thm} \label{enumgeom}
	
	Consider pairs $(S,Q)$ where $S=\bigcup_{i=1}^6 l_i$ is a degenerate sextic  in $\CP^2$  
	 and $Q$ is a rational curve of degree $d$. In general they will meet at $6d$ points. We make the following assumption: 
	
	{\bf Assumption}: For some $z_0$  there exists a degenerate sextic $S_{z_0}=\bigcup l_{i,z_0}$ and a nodal rational curve $Q_{z_0}$ of degree $d$  such that $S_{z_0}$ and $Q_{z_0}$ meet at $ \leq 3d+1$   points of the following type 
	
	\begin{enumerate}
		\item $k$  nodes  of $S_{z_0}$, say $p_{i_1}(z_0),\dots,p_{i_k}(z_0)$. 
		\item $3d-1-k$ other points tangent to $S_{z_0}$ (which could coincide to be of even multiplicity). 
		\item $2$ other points (which could coincide and be of one  of the above types). 
	\end{enumerate}
	
	Then, for {\bf any} degenerate sextic $S_z$ corresponding to a point $z$,
	there exists a nodal rational curves $Q_z$ of degree $d$ meeting the sextic at $3d-1$ 
	points of even multiplicity and two other points. There are two possibilities - either $3d-1-k$ points of tangency and passing through the $k$ nodes $p_{i_1}(z),\dots,p_{i_k}(z)$ or meeting the sextic at $3d-k$ points of tangency and passing thought $k-1$ nodes - and both exist.

\end{thm}

\begin{proof}  
	
	Let $\M_d$ be the space of degree $d$ nodal rational curves in $\CP^2$ and let $\Sigma$ be (an irreducible component of)  the space of sextic curves of the type $\bigcup_{i=1}^6 l_i$.  It is known that the dimension $\dim \M_d=3d-1$. Let $\M=\M_d(k)$ be the subspace of $\M_d$ of rational curves passing through $k$ points in general position. This is of dimension $3d-k-1$.  
	
	Let $\XX$ be the locus in $\M\times \Sigma$ parametrizing pairs $(Q_z,S_z)$ where $Q_z$ meets $S_z$ at $(3d-1)$ double points of which  $k$ are the nodes $p_{i_1}(z),\dots,p_{i_k}(z)$ of $S_z$ and two other points. $\XX$ is non-empty as by assumption $(S_{z_0},Q_{z_0})$ corresponds to a point on $\XX$. There is a morphism 
	$$\pi:\M \times \Sigma \longrightarrow \UH_{6d-2k}$$ 
	$$\pi(Q,S) \longrightarrow (Q \cap S)\backslash \{p_{i_1},\dots, p_{i_k}\}$$
	where $\UH_{6d-2k}$ is the Hilbert scheme of $6d-2k$ points in $\CP^2$. 
	
	Let $Z$ be the stratum of $\UH_{6d-2k}$ parametrizing subschemes of $3d-k-1$ double points $\cup$ $2$ reduced points. Then $\XX$ is just $\pi^{-1}(Z)$. The closure $\bar{Z}$ is an irreducible codimension $(3d-k-1)$ subscheme of $\UH_{6d-2k}$ as it is determined by choosing $3d-k+1$ points - which is $6d-2k+2$ dimensional - and then $3d-k-1$ tangents at $3d-1-k$ of those points - which makes it $9d-2k+1$ dimensional. Further $\UH_{6d-2k}$ is $12d-4k$ dimensional. Since $\UH_{6d-2k}$ is smooth   every irreducible component of $\pi^{-1}(\bar{Z})$ has codimension at most $(3d-k-1)$. 
	
	From our assumption we know that for some $z_0$ in $\Sigma$ there is a pair $(Q_{z_0},S_{z_0})$ in $\pi^{-1}(\bar{Z})$. Equivalently, the local dimension of $\XX \cap (\M \times \{S_{z_0}\})$ is $\geq 0$. Looking at $\pi$ near $(Q_{z_0},S_{z_0})$, since $\dim(\XX) \geq \dim (\Sigma_T)$, the map $\pi:\XX \rightarrow Z$ is flat (\cite{hart},Chapter III, Theorem 9.9) and  the fact that the local fibre dimension is $\geq 0$ at $(Q_{z_0},S_{z_0})$ implies all nearby fibre dimensions are $\geq 0$ . This is equivalent to the fact that there is such a  pair $(Q_z,S_z)$ for every $z$ in a Zariski open neighborhood of $z_0$. Since $\Sigma$ is irreducible, this holds for all pairs $(Q_z,S_z)$ with $z$ in $\Sigma$.

\end{proof}

\begin{rem} In fact, we expect that the dimension is $0$ - namely there are a finite non-zero number of rational curves $Q$ for each $S$  when there is a special rational curve of the same type. There are instances when the special rational curve cannot exist - for instance, one cannot have a nodal cubic meeting six lines only at points of tangency.
\end{rem}

To apply this theorem we observe that the theorem of Birkenhake-Wilhelm give us the special pair $(S_{z_0},Q_{z_0})$ as their theorem says that if an Abelian surface has multiplication by $\ZZ[\sqrt{\Delta}]$ one has such a pair.

Armed with this theorem we can complete the argument. The rational curve $Q_z$ meets $S_z$ at $3d-1$ double points. Then it can be seen that the two remaining points satisfy a quadratic equation. There are two cases. 

Case 1: If $P_1$ and $P_2$ are the two points and they lie on the same line $l$.  Then if $Q$ is given by $F_Q=0$, $F_l:=F_Q|_l$ is a polynomial of degree $d$ in one variable. It has roots at all the double points of $Q \cap l$ and the two remaining points. Dividing out by the linear terms corresponding to the double points gives us the quadratic equation $B$. When the discriminant of $B$ is $0$ the two points $P_1$ and $P_2$ coincide and by the theorem of Birkenhake-Wilhelm we have multiplication by $\ZZ[\sqrt{\Delta'}]$ for some $\Delta'\leq \Delta$. 

Case 2. If $P_1$ and $P_2$ lie on different lines. Say $P_1$ lies on $l_1$ and $P_2$ on $l_2$. Then consider the two polynomials $F_{l_1}$ and $F_{l_2}$. Dividing out by the double roots now gives us two linear polynomials $L_1$ and $L_2$. Then $B=L_1L_2$ is a quadratic polynomial such that if its discriminant is $0$ the two points $P_1$ and $P_2$ coincide at $q_{12}$ and once again the theorem of Birkenhake-Wilhelm applies.

\end{proof}

\section{Graphs}

In order to apply the work of the previous section to obtain equations we need to make choices of $3d-1$ suitable points out of the $3d$ points determined in the theorem of Birkenhake-Willhelm. There are several choices but one has to be slightly careful in choosing the points. For instance, one has to ensure that every line contains $d$ points of multiplicity $2$ - so the points have to either be points of intersection with the other lines or points of even tangency.

For example, for the original theorem of Humbert, we have to choose $6$ points - $5$ double points and one point of tangency. One choice is $q_{12},q_{23},q_{34},q_{45}$ and $q_{51}$ with the point of tangency being at $t_6$ on $l_6$. However, there are clearly other choices. 

As $d$ gets larger it gets increasingly difficult to decide what the possible choices of $3d$ tuples are. Let us call a tuple of $3d$ points {\em admissible} if it corresponds to a case of the theorem of Birkenhake-Wilhelm. 

A convenient way of finding admissible tuples is to relate them to graphs.  We have a correspondence between certain graphs and admissible tuples given as follows:

Let $T$ be an admissible tuple of $3d$ points - a combination of double points $q_{ij}$  and tangents $t_i$ to $T$ where we count the point $r$ times if the multiplicity is $2r$, so $|T|=3d$. To $T$ we associate a graph $\Gamma_T$ as follows: 

\begin{itemize}
\item  $\Gamma_T$ is a graph on $6$ vertices $v_i$, one for every line $l_i$. 

\item  Two vertices $v_i$ and $v_j$ are joined by an edge  $\Leftrightarrow$ $q_{ij} \in T$. 

\item  There is a loop at $v_i$ $\Leftrightarrow$ There is a $t_i$ in $T$. 

\item The condition that every line contains $d$ points this means that $\Gamma_T$ is a $d$-regular graph on $6$ vertices - where loops are counted with valence $2$.

\end{itemize} 

As an example, if $T=\{q_{12},q_{23},q_{34},q_{45},q_{51},t_6\}$, then $\Gamma_T$ is 

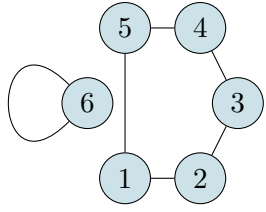
\vspace{\baselineskip}
\begin{center}
	
\begin{tikzpicture}[scale=0.5]
		\node[circle, draw, fill=blue!20] (v1) at (0,0) {1};
		\node[circle, draw, fill=blue!20] (v2) at (2,0) {2};
		\node[circle, draw, fill=blue!20] (v3) at (3,2) {3};
		\node[circle, draw, fill=blue!20] (v4) at (2,4) {4};
		\node[circle, draw, fill=blue!20] (v5) at (0,4) {5};
		\node[circle, draw, fill=blue!20] (v6) at (-1,2) {6};
		
		\draw (v1) -- (v2);
		\draw (v2) -- (v3);
		\draw (v3) -- (v4);
		\draw (v4) -- (v5);
		\draw (v5) -- (v1);

		\draw (v6) to [out=135, in=225, loop] (v6);
		
	\end{tikzpicture}
	\end{center}
	We call this the $(5,1)$ configuration. In general if $T$ has $k$ points and $3d-k$ tangents we call the graph the $(k,3d-k)$ configuration.

For $d=3$ and $T$ corresponding to $9$ double points, there are precisely two $3$-regular graphs on $6$ vertices. One of them is a bipartite  graph $B_{3,3}$ on six vertices. 

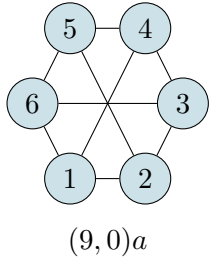
\begin{center}\hypertarget{90a}{}
\begin{tikzpicture}[scale=0.5]
	\node[circle, draw, fill=blue!20] (v1) at (0,0) {1};
	\node[circle, draw, fill=blue!20] (v2) at (2,0) {2};
	\node[circle, draw, fill=blue!20] (v3) at (3,2) {3};
	\node[circle, draw, fill=blue!20] (v4) at (2,4) {4};
	\node[circle, draw, fill=blue!20] (v5) at (0,4) {5};
	\node[circle, draw, fill=blue!20] (v6) at (-1,2) {6};
	
	\draw (v1) -- (v2);
	\draw (v1) -- (v4);
	\draw (v1) -- (v6);
	\draw (v3) -- (v2);
	\draw (v3) -- (v4);
	\draw (v3) -- (v6);
	\draw (v5) -- (v2);
	\draw (v5) -- (v4);
	\draw (v5) -- (v6);
	
\end{tikzpicture}

$(9,0)a$

Bipartite graph with the two sets of vertices being $(v_1,v_3,v_5)$ and $(v_2,v_4,v_6)$
\end{center}

However, computation with the bipartite graphs gives a cubic with discriminant $0$! The reason for this is that  the nine points correspond to the intersection of the two cubics $C_1=l_1l_3l_5$ and $C_2=l_2l_4l_6$. The Cayley-Bacharach theorem  says that in this case any cubic passing through  $8$ of the  points will necessarily pass through  the $9^{th}$. So if we consider a rational cubic passing through $8$ of these points it will always pass through the remaining point. 

The second $(9,0)$ configuration is the following. 

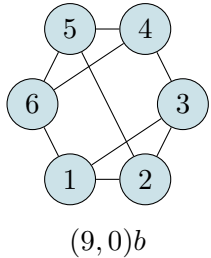
\begin{center}\hypertarget{90b}{}
	
\begin{tikzpicture}[scale=0.5]
	\node[circle, draw, fill=blue!20] (v1) at (0,0) {1};
	\node[circle, draw, fill=blue!20] (v2) at (2,0) {2};
	\node[circle, draw, fill=blue!20] (v3) at (3,2) {3};
	\node[circle, draw, fill=blue!20] (v4) at (2,4) {4};
	\node[circle, draw, fill=blue!20] (v5) at (0,4) {5};
	\node[circle, draw, fill=blue!20] (v6) at (-1,2) {6};
	
	\draw (v1) -- (v2);
	\draw (v1) -- (v3);
	\draw (v1) -- (v6);
	\draw (v2) -- (v5);
	\draw (v2) -- (v3);
	\draw (v3) -- (v4);
	\draw (v5) -- (v4);
	\draw (v5) -- (v6);
	\draw (v6) -- (v4);

\end{tikzpicture}

$(9,0)b$

\end{center}
This corresponds to 
$$T_{(9,0)}=(q_{12},q_{13},q_{16},q_{23},q_{25},q_{34},q_{45},q_{46},q_{56})$$
Of course a relabeling of the vertices will give another admissible tuple.

\subsection{Classification of Graphs}
 
There may be several non-isomorphic graphs corresponding to the same tuple $(d,k)$. In this section we classify all the possibilities for $2$ and $3$ regular graphs on $6$ vertices. 

We use  the built-in graph modules in SAGE.  In order to classify, we shall ignore labeling. This reduces to looking at the degree sequence for each class of graphs. Denote the degree sequence of a graph on $6$ vertices as $[d_1,d_2,..,d_6]$ where $d_i$ is the degree of the vertex corresponding to singular line $l_i$, $1\leq i\leq 6$.
 
\subsubsection{2-Regular Graphs}

For a $2$-regular graph with one loop on one of the vertices, choose a vertex and assign a loop. The remaining $5$ vertices forms a connected $2$ regular graph. Consider the following code.
\begin{verbatim}
	# List to store all unique 2-regular graphs on 5 vertices
	two_regular_graphs = []
	
	# Generate all connected graphs with 5 vertices
	for G in graphs.nauty_geng("5 -c"):
	# Check if it's 2-regular
	if all(deg == 2 for deg in G.degree()):
	two_regular_graphs.append(G)
	
	# Output
	print(f"Total 2-regular graphs on 5 vertices: {len(two_regular_graphs)}")
	for i, g in enumerate(two_regular_graphs):
	print(f"\nGraph {i+1}:")
	show(g)
\end{verbatim}
This gives us that there is only one $2$-regular graph on $5$ vertices upto isomorphism.

If there are $2$ loops on two isolated vertices, the remaining $4$ vertices must form a connected $2$ regular graph. We can modify the above code, to see that there is only one $2$-regular graph on $4$ vertices. And similarly, we find only one $2$-regular graph on $3$ vertices.

\begin{figure}[H]
	\centering
	\begin{subfigure}{0.3\textwidth}
		\centering
		\begin{tikzpicture}[scale=0.5]
			\node[circle, draw, fill=blue!20] (v1) at (0,0) {1};
			\node[circle, draw, fill=blue!20] (v2) at (2,0) {2};
			\node[circle, draw, fill=blue!20] (v3) at (3,2) {3};
			\node[circle, draw, fill=blue!20] (v4) at (2,4) {4};
			\node[circle, draw, fill=blue!20] (v5) at (0,4) {5};
			\node[circle, draw, fill=blue!20] (v6) at (-1,2) {6};
			
			\draw (v1) -- (v2);
			\draw (v1) -- (v5);
			\draw (v2) -- (v3);
			\draw (v3) -- (v4);
			\draw (v5) -- (v4);

			\draw (v6) to [out=135, in=225, loop] (v6);
		\end{tikzpicture}
		\caption{(5,1)}
		\label{51}
	\end{subfigure}
	\hfill
	\begin{subfigure}{0.3\textwidth}
		\centering
		\begin{tikzpicture}[scale=0.5]
			\node[circle, draw, fill=blue!20] (v1) at (0,0) {1};
			\node[circle, draw, fill=blue!20] (v2) at (2,0) {2};
			\node[circle, draw, fill=blue!20] (v3) at (3,2) {3};
			\node[circle, draw, fill=blue!20] (v4) at (2,4) {4};
			\node[circle, draw, fill=blue!20] (v5) at (0,4) {5};
			\node[circle, draw, fill=blue!20] (v6) at (-1,2) {6};
			
			\draw (v1) -- (v2);
			\draw (v1) -- (v4);
			\draw (v2) -- (v3);
			\draw (v3) -- (v4);

			\draw (v6) to [out=135, in=225, loop] (v6);
			\draw (v5) to [out=135, in=225, loop] (v5);
		\end{tikzpicture}
		\caption{(4,2)}
		\label{42}
	\end{subfigure}
	\hfill
	\begin{subfigure}{0.3\textwidth}
		\centering
		\begin{tikzpicture}[scale=0.5]
			\node[circle, draw, fill=blue!20] (v1) at (0,0) {1};
			\node[circle, draw, fill=blue!20] (v2) at (2,0) {2};
			\node[circle, draw, fill=blue!20] (v3) at (3,2) {3};
			\node[circle, draw, fill=blue!20] (v4) at (2,4) {4};
			\node[circle, draw, fill=blue!20] (v5) at (0,4) {5};
			\node[circle, draw, fill=blue!20] (v6) at (-1,2) {6};
			
			\draw (v1) -- (v2);
			\draw (v1) -- (v3);
			\draw (v2) -- (v3);

			\draw (v6) to [out=135, in=225, loop] (v6);
			\draw (v5) to [out=135, in=225, loop] (v5);
			\draw (v4) to [out=45, in=-45, loop] (v4);
		\end{tikzpicture}
		\caption{(3,3)}
		\label{33}
	\end{subfigure}
	\caption{Isomorphism classes of 2 regular graphs with loops}
\end{figure}
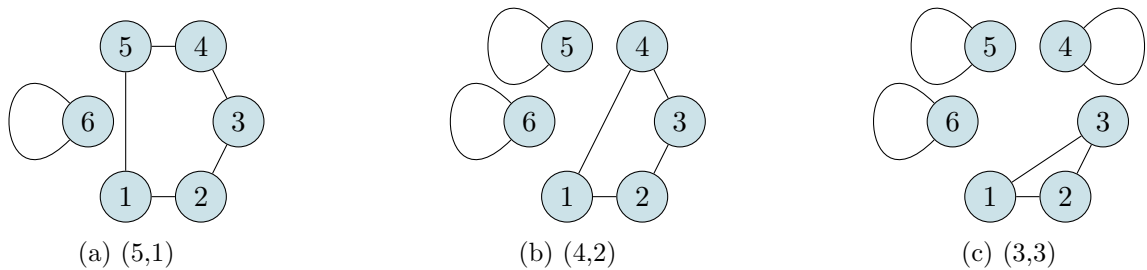

\subsubsection{3-Regular Graphs}

{\bf Case 1: (9,0)} A slight modification of the code above gives us that there are two $3$-regular graphs on $6$ vertices upto isomorphism. One of them is bipartite and the other is not. Note that all graphs in this situation are connected. The two graphs above \hyperlink{90a}{bipartite} and \hyperlink{90b}{non-bipartite}.

\noindent{\bf Case 2: (8,1)} Choose a vertex and assign a loop to it. This vertex must connect to some other vertex via a single edge. Thus, we want to classify graphs upto isomorphism on the remaining $5$ vertices such that $4$ of them have degree $3$ and one vertex has degree $2$. The corresponding degree for the $5$ vertices is $[3, 3, 3, 3, 2]$. Consider the code
\begin{verbatim}
	# Desired degree sequence
	target_degrees = sorted([3, 3, 3, 3, 2])
	
	# Container for graphs matching the degree condition
	matching_graphs = []
	
	# Generate all graphs with 5 vertices
	for G in graphs.nauty_geng("5"):
	# Check if graph has 7 edges and matching degree sequence
	if G.size() == 7 and sorted(G.degree()) == target_degrees:
	# Check for isomorphism to avoid duplicates
	if not any(G.is_isomorphic(H) for H in matching_graphs):
	matching_graphs.append(G)
	
	# Output results
	print(f"Total graphs with degrees [3, 3, 3, 3, 2]: {len(matching_graphs)}")
	for i, g in enumerate(matching_graphs):
	print(f"\nGraph {i+1}:")
	show(g)
\end{verbatim}
This shows that there is only one such graph. All graphs in this situation are connected.

\begin{center}\label{81}
	\begin{tikzpicture}[scale=0.5]
		\node[circle, draw, fill=blue!20] (v1) at (0,0) {1};
		\node[circle, draw, fill=blue!20] (v2) at (2,0) {2};
		\node[circle, draw, fill=blue!20] (v3) at (3,2) {3};
		\node[circle, draw, fill=blue!20] (v4) at (2,4) {4};
		\node[circle, draw, fill=blue!20] (v5) at (0,4) {5};
		\node[circle, draw, fill=blue!20] (v6) at (-1,2) {6};
		
		\draw (v1) -- (v2);
		\draw (v1) -- (v5);
		\draw (v1) -- (v6);
		\draw (v2) -- (v3);
		\draw (v3) -- (v4);
		\draw (v5) -- (v4);
		\draw (v5) -- (v3);
		\draw (v2) -- (v4);

		\draw (v6) to [out=135, in=225, loop] (v6);
	\end{tikzpicture}
	
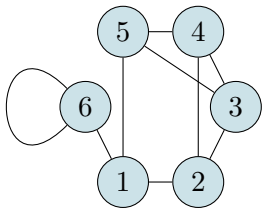
	Only possible graph of type $(8,1)$.
\end{center}

\noindent{\bf Case 3: (7,2)} Choose two vertices and assign them one loop each. If the two vertices containing a loop each are connected, then the remaining $4$ vertices must be connected and $3$-regular. In this situation, there is only one such graph. Each of the four vertices connect to other three vertices, this is the complete graph on four vertices. Thus there is only one such graph.

If two of the vertices with loops connect a single vertex among the remaining four, then the resulting graph can not be $3$-regular. There will always be a vertex with degree $2$.

Finally, if two vertices containing loops are adjacent to two distinct vertices among the remaining $4$, then we want to classify graphs on $4$ vertices with degree sequence $[3,3,2,2]$. Consider the following code. 
\begin{verbatim}
	# Target degree sequence
	target_degrees = sorted([3, 3, 2, 2])
	
	# Store matching graphs
	matching_graphs = []
	
	# Generate all graphs with 4 vertices
	for G in graphs.nauty_geng("4"):
	if G.size() == 5 and sorted(G.degree()) == target_degrees:
	# Avoid isomorphic duplicates
	if not any(G.is_isomorphic(H) for H in matching_graphs):
	matching_graphs.append(G)
	
	# Output the results
	print(f"Total graphs with degrees [3, 3, 2, 2]: {len(matching_graphs)}")
	for i, g in enumerate(matching_graphs):
	print(f"\nGraph {i+1}:")
	show(g)
\end{verbatim}
The above code gives us that there is only one such graph. To conclude, there are two $3$-regular graphs with $2$ loops, one is connected and the other one has two connected components. The graphs are depicted below.

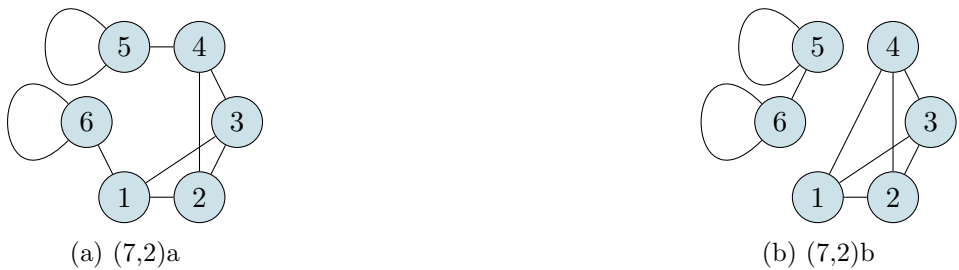
\begin{figure}[H]
	\centering
	\begin{subfigure}{0.45\textwidth}
		\centering
		\begin{tikzpicture}[scale=0.5]
			\node[circle, draw, fill=blue!20] (v1) at (0,0) {1};
			\node[circle, draw, fill=blue!20] (v2) at (2,0) {2};
			\node[circle, draw, fill=blue!20] (v3) at (3,2) {3};
			\node[circle, draw, fill=blue!20] (v4) at (2,4) {4};
			\node[circle, draw, fill=blue!20] (v5) at (0,4) {5};
			\node[circle, draw, fill=blue!20] (v6) at (-1,2) {6};
			
			\draw (v1) -- (v2);
			\draw (v1) -- (v3);
			\draw (v2) -- (v3);
			\draw (v2) -- (v4);
			\draw (v3) -- (v4);
			
			\draw (v5) -- (v4);
			\draw (v6) -- (v1);

			\draw (v6) to [out=135, in=225, loop] (v6);
			\draw (v5) to [out=135, in=225, loop] (v5);

		\end{tikzpicture}
		\caption{(7,2)a}
		
	\end{subfigure}
	\hfill
	\begin{subfigure}{0.45\textwidth}
		\centering
		\begin{tikzpicture}[scale=0.5]
			\node[circle, draw, fill=blue!20] (v1) at (0,0) {1};
			\node[circle, draw, fill=blue!20] (v2) at (2,0) {2};
			\node[circle, draw, fill=blue!20] (v3) at (3,2) {3};
			\node[circle, draw, fill=blue!20] (v4) at (2,4) {4};
			\node[circle, draw, fill=blue!20] (v5) at (0,4) {5};
			\node[circle, draw, fill=blue!20] (v6) at (-1,2) {6};
			
			\draw (v1) -- (v2);
			\draw (v1) -- (v3);
			\draw (v1) -- (v4);
			\draw (v2) -- (v3);
			\draw (v2) -- (v4);
			\draw (v3) -- (v4);
			
			\draw (v5) -- (v6);

			\draw (v6) to [out=135, in=225, loop] (v6);
			\draw (v5) to [out=135, in=225, loop] (v5);
			
		\end{tikzpicture}

		\caption{(7,2)b}
	\end{subfigure}
	
	\caption{Possible graphs of type $(7,2)$}
	\label{72}
\end{figure}

\noindent{\bf Remaining cases} 

In all the remaining cases - $(6,3), (5,4), (4,5)$ and $(3,6)$ there is precisely one graph for the each admissible type. Here is the argument for $(6.3)$. The rest are similar. 

Choose three vertices to assign loops to them. These three vertices can not be adjacent to a single vertex among the remaining three, then the remaining two vertices must have degree $2$.

The three vertices containing loops can not be adjacent to each other. If any two of these vertices are adjacent and the remaining vertex with a loop is adjacent to a vertex without loop, then necessarily this vertex will have degree two. Two of the vertices containing loops can not be adjacent to the same vertex among the ones without loop, then the resulting graph can not be $3$-regular.

So the only possibility is that the three vertices containing loops must be adjacent to the three remaining vertices each. Thus we want to classify graphs on three vertices with degree sequence $[2,2,2]$. This is the complete graph on three vertices. To conclude, there is only one $3$-regular graph on six vertices containing three loops. The remaining cases can be tackled in a similar manner. 

\begin{figure}[H]
	\centering
	\begin{subfigure}{0.45\textwidth}
		\centering
	\begin{tikzpicture}[scale=0.5]
		\node[circle, draw, fill=blue!20] (v1) at (0,0) {1};
		\node[circle, draw, fill=blue!20] (v2) at (2,0) {2};
		\node[circle, draw, fill=blue!20] (v3) at (3,2) {3};
		\node[circle, draw, fill=blue!20] (v4) at (2,4) {4};
		\node[circle, draw, fill=blue!20] (v5) at (0,4) {5};
		\node[circle, draw, fill=blue!20] (v6) at (-1,2) {6};
		
		\draw (v1) -- (v2);
		\draw (v1) -- (v3);
		\draw (v2) -- (v3);
		\draw (v2) -- (v5);
		\draw (v1) -- (v6);
		\draw (v3) -- (v4);

		\draw (v6) to [out=135, in=225, loop] (v6);
		\draw (v5) to [out=135, in=225, loop] (v5);
		\draw (v4) to [out=45, in=-45, loop] (v4);
	\end{tikzpicture}
	\label{63}
	\caption{Only graph of type $(6,3)$} 
\end{subfigure}
\hfill
\begin{subfigure}{0.45 \textwidth}
\centering

	\begin{tikzpicture}[scale=0.5]
		\node[circle, draw, fill=blue!20] (v1) at (0,0) {1};
		\node[circle, draw, fill=blue!20] (v2) at (2,0) {2};
		\node[circle, draw, fill=blue!20] (v3) at (3,2) {3};
		\node[circle, draw, fill=blue!20] (v4) at (2,4) {4};
		\node[circle, draw, fill=blue!20] (v5) at (0,4) {5};
		\node[circle, draw, fill=blue!20] (v6) at (-1,2) {6};
		
		\draw (v1) -- (v2);
		\draw (v1) -- (v5);
		\draw (v1) -- (v6);
		\draw (v2) -- (v3);
		\draw (v2) -- (v4);

		\draw (v6) to [out=135, in=225, loop] (v6);
		\draw (v5) to [out=135, in=225, loop] (v5);
		\draw (v4) to [out=45, in=-45, loop] (v4);
		\draw (v3) to [out=45, in=-45, loop] (v3);
	\end{tikzpicture}
\caption{Only graph of type (5,4)}
\label{54}
\end{subfigure}

\end{figure}

\begin{figure}[H]
\centering
\begin{subfigure}{0.45 \textwidth}
\centering

	\begin{tikzpicture}[scale=0.5]
		\node[circle, draw, fill=blue!20] (v1) at (0,0) {1};
		\node[circle, draw, fill=blue!20] (v2) at (2,0) {2};
		\node[circle, draw, fill=blue!20] (v3) at (3,2) {3};
		\node[circle, draw, fill=blue!20] (v4) at (2,4) {4};
		\node[circle, draw, fill=blue!20] (v5) at (0,4) {5};
		\node[circle, draw, fill=blue!20] (v6) at (-1,2) {6};
		
		\draw (v1) -- (v2);
		\draw (v1) -- (v3);
		\draw (v1) -- (v4);
		\draw (v5) -- (v6);

		\draw (v6) to [out=135, in=225, loop] (v6);
		\draw (v5) to [out=135, in=225, loop] (v5);
		\draw (v4) to [out=45, in=-45, loop] (v4);
		\draw (v3) to [out=45, in=-45, loop] (v3);
        \draw (v2) to [out=45, in=-45, loop] (v2);

	\end{tikzpicture}

\caption{Only graph of type $(4,5)$}
\label{45}
\end{subfigure}
\hfill
\begin{subfigure}{0.45 \textwidth} 
\centering
	\begin{tikzpicture}[scale=0.5]
		\node[circle, draw, fill=blue!20] (v1) at (0,0) {1};
		\node[circle, draw, fill=blue!20] (v2) at (2,0) {2};
		\node[circle, draw, fill=blue!20] (v3) at (3,2) {3};
		\node[circle, draw, fill=blue!20] (v4) at (2,4) {4};
		\node[circle, draw, fill=blue!20] (v5) at (0,4) {5};
		\node[circle, draw, fill=blue!20] (v6) at (-1,2) {6};
		
		\draw (v1) -- (v2);
		\draw (v3) -- (v4);
		\draw (v5) -- (v6);

		\draw (v1) to [out=135, in=225, loop] (v1);
		\draw (v2) to [out=45, in=-45, loop] (v2);
		\draw (v3) to [out=45, in=-45, loop] (v3);
		\draw (v4) to [out=45, in=-45, loop] (v4);
		\draw (v5) to [out=135, in=225, loop] (v5);
		\draw (v6) to [out=135, in=225, loop] (v6);

	\end{tikzpicture}
\label{36}
\caption{Only graph of type $(3,6)$} 
\end{subfigure}

\end{figure}

\section{Computations}

In this section we use ideas of the previous sections to compute the equations. The idea is the following. We want to find a rational curve of degree $d$ passing through $k$ points and meeting the lines of the sextic $S$ at $3d-k-1$ points of even multiplicity - where we count points of multiplicity $2r$ $r$ times. We know such a curve always exists and it will meet $S$ at two other points. The condition we want is that the two points coincide - so that $S$ and the rational degree $d$ curve meet at $3d$ double points. 

For $d>2$ it seems to be computationally easier to do the following. Find conditions on the degree $d$ polynomials so that the corresponding curves pass  through $3d$ double points. Then impose further conditions so that the curve is rational.  

A homogeneous equation of degree $d$ in $3$ variables has $N=\binom{d+2}{2}=\frac{(d+1)(d+2)}{2}$ coefficients, so the space of degree $d$ curves is $N-1$ dimensional. 

The condition that such a curve passes through a point determines a linear relation among the coefficients and so the dimension drops by $1$. The condition that points are in `general position' translates to the fact that the corresponding linear relations are linearly independent. Hence the space of degree $d$ curves passing through $k$ points in general position is of dimension $N-k-1$. 

The condition that it is tangent to a line also drops the dimension by $1$. Hence the space of degree $d$ curves passing through $k$ points and tangent to $3d-k$ lines is of dimension $\binom{d+2}{2}=\frac{(d+1)(d+2)}{2}-(k+(3d-k))=\frac{(d-1)(d-2)}{2}$. Note that this is the genus of a smooth curve of degree $d$. We then need to impose conditions that the curves are rational - that is, have geometric genus $0$. This adds a further $g$ conditions resulting in a $0$ dimensional set of rational curves passing through the $3d$ double points. We see that this imposes conditions on the coefficients - which are computable in terms of $a_1, a_2$ and $a_3$. 

We can equally well first perform the blowing up to get conditions under which a degree $d$ curve is rational. This will a $3d-1$ dimensional space. Then use the other conditions to ensure that the curves pass through the required number of points and are tangent to the required number of lines.

The algorithm works particularly well in degree $3$. If $d=3$ then $\binom{d+2}{2}-1=3d=9$ so there exist a unique cubic curve through $9$ points. However, in general it  not {\em rational}. In order to get rationality we have to impose one more condition - that the cubic is singular. This is guaranteed by the vanishing of the discriminant of the cubic. Hence the algorithm is --

\vspace{\baselineskip}

\noindent  Step 1. Compute the equation of the {\em possibly smooth cubic} through the nine points of an admissible tuple. The coefficients will be determined by $a_1, a_2$ and $a_3$. 

\noindent Step 2. Compute the discriminant of this cubic to get a polynomial $F$ in $a_1, a_2$ and $a_3$. When this is $0$ the corresponding Jacobian will have real multiplication.

We recall some necessary background and theory and then illustrate the algorithm using SAGE.
\subsection{Plane Curves}
\label{planecurves}

The material in this section is standard and can be found in an textbook on algebraic geometry or algebraic curves, for instance Hartshorne \cite{hart} or Fulton \cite{fult}.

A {\em plane curve}  $C$ of degree $d$  is the zero locus in $\CP^2$ of an equation of the following form
\begin{align}
\label{eq1}
f:\sum_{i+j+k=d} c_{ijk} x^{i}y^{j}z^{j}=c_{d00}x^d+c_{d-110}x^{d-a}y^1+\dots  +c_{00d}z^{d}
\end{align}
where $c_{ijk}\in\C$ for $i=1,2,...,10$.
$$C=V(f)=\{[x_0,y_0,z_0] \in \CP^2|f(x_0,y_0,z_0)=0\}$$
Let  \[f_s:=\frac{\partial f}{\partial s}\] be the formal partial derivative of $f$ with respect to the variable $s$. 
\begin{thm}[Euler's Identity]
\label{th5.1}
    Let $f\in\C[x_1,...,x_n]$ be a non-constant homogeneous polynomial of degree $d>0$. Then the following identity
    \begin{align*}
        d \cdot f = x_1f_{x_1}+\cdots+x_nf_{x_n}
    \end{align*}
    holds.
\end{thm}
\begin{proof}
    The formal partial derivative operator is linear with respect to addition, that is, for any two polynomials $g_1,g_2\in\C[x_1,...,x_n]$, $(g_1+g_2)_{x_i}={g_1}_{x_i}+{g_2}_{x_i}$. Hence it suffices to check this identity on a monomial of degree $d$.\\
    Let $f=Cx_1^{i_1}\cdots x_n^{i_n}$ with $i_1+\cdots i_n=d$ and some constant $C\in\C$. With respect to some variable $x_j$, $x_j\cdot f_{x_j}=i_j\cdot Cx_1^{i_1}\cdots x_n^{i_n}=i_j\cdot f$. Summing over $j=1,...,n$ we get $\sum_{j=1}^{n}x_jf_{x_j}=\sum_{j=1}^{n}i_jf=d\cdot f$
\end{proof}
We need this result for the following simple corollary.

\begin{cor}
\label{cor5.2}
    If all partial derivatives of a non-constant polynomial $f\in\C[x,y,z]$ vanish at some $p\in\CP^2(\C)$ then necessarily $f(p)=0$ as well.
\end{cor}

\noindent The {\em geometric genus} or {\em topological genus} of a plane curve is defined to be $g=g_g(C)=\dim_{\C} H^0(C,\Omega^1_{C})$. The {\em arithmetic genus} is defined to be $g_a(C)=\dim_{\C} H^1(C,\OO_C)$. While the arithmetic genus is constant in families the geometric genus could change. For a smooth curve $g_a(C)=g_g(C)$ but in general $g_a(C) \geq g_g(C)$ 

We have the following degree-genus formula for planar curves 

\begin{thm}
\label{degree-genus}
    Let $C$ be an irreducible planar curve defined over $\C$ of degree $d$ in $\CP^2(\C)$ with only ordinary singularities.

    \begin{align}
    \label{deggen}
        g_g(C)=\frac{(d-1)(d-2)}{2}-\sum_{x \in Sing(C)} \frac{1}{2}r_x(r_x-1)
    \end{align}
    where $r_x$ is the order of the singularity at $x$.
\end{thm}
\begin{proof} This can be found in numerous places, for instance  Corollary 2.4, \cite{hart} \end{proof}

A curve $C$ defined over $\C$ is called \emph{rational} if the $g_g(C)=0$. From the formula \eqref{deggen} curves of degree $1$ and $2$ are necessarily rational. If $C$ is smooth of degree $3$ it is an elliptic curve - as it always has a rational point over $\C$. However, a cubic curve with one ordinary node ($r=2$) or a cusp is also rational. Since the genus is non-negative, an irreducible cubic curve can have at most one singular point. 

Another classical theorem that we will need is the following:
\begin{thm}[Bezout]
\label{bezout}
    Let $C_1$ and $C_2$ be plane curves of degree $d_1$ and $d_2$ respectively. Then the number of points of intersection 
    $$\#\{C_1\cap C_2\}=d_1d_2$$
    counting multiplicities.
\end{thm}

\begin{proof}
(see corollary 7.8, \cite{hart}).

\end{proof}

Let $C_1$ be of degree $d_1=3$ and $C_2$ be of degree $d_2=1$ (a line). Then $\#\{C_1\cap C_2\}=3$. We have the following three possible configurations for the points of intersections of $C_1$ and $C_2$
\begin{enumerate}
    \item All three points of intersection are transversal
    \item One point of transversal intersection and another point of multiplicity $2$
    \item All three points coincide
\end{enumerate}

Recall that for our calculations we need to find the cubic passing through $8$ points of multiplicity $2$ on the sextic $S=\bigcup_{i=1}^6 l_i$. Since we require that the cubic and sextic meet only at points of even multiplicity or transversally  at two other points, Case 3 cannot occur.

\subsection{Resultants}
\label{resultants}
The material of this section can be found in literature in many places. For example, Chapter 12 and Chapter 13 of \cite{GeKZ} and also Chapter 3 of \cite{CLO}.

Let $m\geq 1$ be an integer and $f_0,f_1,...,f_m\in\C[t_0,...,t_m]$ be homogeneous polynomials in $m+1$ variables and of degrees $d_0,...,d_m$ respectively. The \emph{resultant} of these $m+1$ forms, denoted $\res_{d_0,...,d_m}(f_0,...,f_m)$ (or simply $\res(f_0,...,f_m)$), is an irreducible polynomial over $\Z$ in the coefficients of the $f_i$'s. This  has the property that 
$$\res_{d_0,...,d_m}(f_0,...,f_m)=0 \Leftrightarrow \{f_0,...,f_m\} \text{ have a common zero in } \C^{m+1}\backslash \{0\}$$

One can uniquely define the resultant by requiring that $\res_{d_0,...,d_m}(t_0^{d_0},...,t_m^{d_m})=1$ (cf. pg. 427, Chapter 13, \cite{GeKZ}). We note the following result.
\begin{prop}
\label{resultantdeg}
Let $m\geq 1$ be an integer and $f_0,f_1,...,f_m\in\C[t_0,...,t_m]$ be homogeneous polynomials in $m+1$ variables and of degrees $d_0,...,d_m$ respectively. The resultant is a homogeneous polynomial in the coefficients of each $f_i$ of degree $d_0 d_1 \cdots d_{i-1} d_{i+1}\cdots d_m$. In particular, the total degree of the resultant is $d_0\cdots d_m(d_0^{-1}+\cdots+d_m^{-1})$.
\end{prop}

\begin{proof} Proposition 1.1, pg. 427, Chapter 13 of \cite{GeKZ}
\end{proof}

For two affine polynomials in one variable defined over an algebraically field of characteristic $0$ we have the following expression for the resultant. 
\begin{prop}
\label{resultantoftwopoly}
    Let $K$ be an algebraically closed field of characteristic $0$ and $f,g\in K[X]$ polynomials of degree $m,n$ respectively.
    \[f(x)=a_0+a_1x+\cdots a_mx^m\] and
    \[g(x)=b_0+b_1x+\cdots+b_nx^n\]
    Then the resultant is 
    \begin{align}
    \label{eq3}
       \res(f,g)= \det\begin{bmatrix}
    a_0 & a_1 & ... & ... & ... & ... & a_m & \\
     & a_0 & a_1 & ... & ... & ... & a_{m-1} & a_m & \\
    &  & a_0 & ... & ... & ... & a_{m-2} & a_{m-1} & a_m & \\
    & ....... &  &  & \\
    & ....... & \\
    b_0 & b_1 & b_2 & ... & ... & b_n & \\
     & b_0 & b_1 & ... & ... & b_{n-1} & b_n & \\
    &  & b_0 & ... & ... & b_{n-2} & b_{n-1} & b_n & \\
    & ....... &
    \end{bmatrix} 
    \end{align}
    Here there are $n$ rows of $a_i's$ and $m$ rows of $b_j's$ - so the matrix is a square matrix of rank $m+n$. 
    
\begin{center}    
    $f$ and $g$ have a common zero $\Leftrightarrow$ $\res(f,g)=0$.
    
    \end{center}
\end{prop}
\begin{proof}\cite{CLO} pg. 77, Chapter 3.\end{proof}

We can homogenize the two polynomials in Proposition \ref{resultantoftwopoly} and obtain a similar condition for  homogeneous polynomials to have a non-trivial zero - that is, a zero in $\CP^1(K)$. This is discussed in the next section with $g(x) = \frac{d f}{d x} = f_x(x)$.

\subsubsection{Discriminant of a homogeneous polynomial in two variables}
\label{bivariatepolynomials}
Consider a binary form $f\in\C[x,y]$ of degree $d>0$
\begin{align}
\label{eq4}
    f(x,y)=\sum_{i}^d  c_i x^{d-i}y^a
\end{align}
The {\em discriminant of $f$}, $\disc(f)$  is defined to be the resultant of the two formal partial derivatives of $f$, 
$$\disc(f)=\res_{d-1,d-1}(f_x,f_y)$$
The following is due to Sylvester.

\begin{prop}[Sylvester] Let $f$ be as above, the discriminant has the following expression. 
\label{sylvester}
\begin{align}
   \disc(f)=\det\begin{bmatrix}
        c_1 & 2c_2 & \cdots & \cdots & dc_{d} & \dots \\
         & c_1 & 2c_2 & \cdots & (d-1)c_{d-1} & dc_{d} & \dots \\
         \dots \\
         \dots \\
        c_0 & 2c_1 & \cdots & \cdots & (d-1)c_{d-1} & \dots \\
         & c_0 & 2c_1 & \cdots & (d-2)c_{d-2} & (d-1)c_{d-1} \dots \\
         \dots \\
         \dots
    \end{bmatrix}
\end{align}
\end{prop}
\begin{proof} \cite{GeKZ}, eq 2.9, Chapter 2, pg 60. This is a $(2d-2) \times (2d-2)$ matrix. \end{proof}

	We apply this to the following situation. If $f$ is a homogeneous polynomial of degree $d$ in $3$ variables and $g$ is a homogeneous polynomial of degree $1$, then $h=f\circ g$ is a polynomial in two variables which determines the points of intersection of $V(f)$ with the line  $l=V(g) \simeq \CP^1$. Hence 
	$$l \text{ is tangent to } V(f) \Leftrightarrow \disc(h)=0$$

\subsubsection{Discriminant of a plane cubic curve}
Let $f\in\C[x,y,z]$ be a homogeneous equation of degree $3$ defining a planar cubic curve. Let us denote the curve by $C:=\{[x:y:z]\in\CP^2(\C)\mid f(x,y,z)=0\}$. 

Following the discussion after Theorem \hyperref[degree-genus]{\ref{degree-genus}}, $C$ is \emph{rational} if and only if it has a singular point at some $p\in C$. Equivalently, The tangent space of the curve at a point $p$ is more than one dimensional. The tangent space at $p$ has the equation
\begin{align*}
    xf_x(p)+yf_y(p)+zf_z(p)=0
\end{align*}
So the tangent space is not well-defined if and only if all three partial derivatives vanish at $p\in\CP^2(\C)$, i.e. $f(p)=f_x(p)=f_y(p)=f_z(p)=0$. Observe that, the point $p\in\CP^2(\C)$ is a common root for each of the partial derivatives of $f$. Note that the partial derivatives $f_x,f_y,f_z$ are homogeneous of degree $2$. We have the following.
\begin{prop}
\label{rationalcubic}
    Let $C$ be a planar cubic curve given by the zero locus of $f\in\C[x,y,z]$. Then 
    
    $$C \text { is rational }\Leftrightarrow \res_{2,2,2}(f_x,f_y,f_z)=0$$
\end{prop}
\begin{proof}
    If $C$ is rational, by the above discussion, there exists a point $p$ on $C$ such that $f_x(p)=f_y(p)=f_z(p)=0$, hence the resultant $\res_{2,2,2}(f_x,f_y,f_z)$ also must vanish since $p$ is a common zero of $f_x,f_y,f_z$.
    
    Conversely, if the resultant vanishes, then there exists a non-trivial zero of $f_x,f_y,f_z$, say $p\in\CP^2(\C)$. We need to show that this point lies on the curve $C$ as well, but this is true by corollary \hyperref[cor5.2]{5.2}. Hence $p$ is a singular point on $C$. Hence $C$ is rational.
\end{proof}

This polynomial $\res_{2,2,2}(f_x,f_y,f_z)$ is called the discriminant of $f$, $\disc(f)$ or the discriminant of the cubic curve $C$, $\disc(C)$ . Salmon gave the following expression for the discriminant

\begin{thm}[Salmon]

Let $f$ be a homogeneous polynomial of degree $3$ in $\C[x,y,z]$ defining a cubic curve $C$ in $\CP^2$.
$$f(x,y,z)=\sum_{i+j+k=3} c_{ijk}x^iy^jz^k$$
Let the formal partial derivatives of $f$ be
\begin{align*}
    f_x=3c_{300}x^2 + 2c_{210}xy + c_{120}y^2 + 2c_{201}xz + c_{111}yz + c_{102}z^2
    \\
    f_y=c_{210}x^2 + 2c_{120}xy + 3c_{030}y^2 + c_{111}xz + 2c_{021}yz + c_{012}z^2
    \\
    f_z=c_{201}x^2 + c_{111}xy + c_{021}y^2 + 2c_{102}xz + 2c_{012}yz + 3c_{003}z^2
\end{align*}
Let $F$ be the determinant of the Jacobian matrix
\begin{align*}
    F=\det\begin{bmatrix}
        \frac{\partial f_x}{\partial x} & \frac{\partial f_x}{\partial y} & \frac{\partial f_x}{\partial z}\\
        \frac{\partial f_y}{\partial x} &
        \frac{\partial f_y}{\partial y} & \frac{\partial f_y}{\partial z}\\ 
        \frac{\partial f_z}{\partial x} &
        \frac{\partial f_z}{\partial y} & \frac{\partial f_z}{\partial z}
    \end{bmatrix}
\end{align*}
Note that $F\in\C[x,y,z]$ is homogeneous of degree $3$ because each entry of the matrix is linear. Take the following formal partial derivatives
\begin{align*}
    F_x=b_{11}x^2+b_{12}y^2+b_{13}z^2+b_{14}xy+b_{15}xz+b_{16}yz\\
    F_y=b_{21}x^2+b_{22}y^2+b_{23}z^2+b_{24}xy+b_{25}xz+b_{26}yz\\
    F_z=b_{31}x^2+b_{32}y^2+b_{33}z^2+b_{34}xy+b_{35}xz+b_{36}yz
\end{align*}
The coefficients $b_{ij}$ are determined by the $c_{ijk}$ for instance 

$$	b_{11}=-24c_{120}c_{201}^2 + 24c_{111}c_{201}c_{210} - 24c_{102}c_{210}^2 - 18c_{111}^2c_{300} + 72c_{102}c_{120}c_{300}$$

One then has
\begin{align}
\label{eq6}
    \disc(f)=\disc(C)=\res_{2,2,2}(f_x,f_y,f_z)=\det\begin{bmatrix}
    3c_{300} & c_{120} & c_{102} & 2c_{210} & 2c_{201} & c_{111}\\
    c_{210} & 3c_{030} & c_{012} & 2c_{120} & c_{111} & 2c_{021}\\
    c_{201} & c_{021} & 3c_{003} & c_{111} & 2c_{102} & 2c_{012}\\
    b_{11} & b_{12} & b_{13} & b_{14} & b_{15} & b_{16}\\
    b_{21} & b_{22} & b_{23} & b_{24} & b_{25} & b_{26}\\
    b_{31} & b_{32} & b_{33} & b_{34} & b_{35} & b_{36}
    \end{bmatrix}
\end{align}

\end{thm}
\begin{proof}\cite{CLO} Chapter 3, eq.(2.8), pg. 89.\end{proof}

The resultant is well-defined up to multiplication by a non-zero rational number.

\subsection{Vertex covers and products of lines}
\label{lines}

For the purpose of computation it is useful to find a generating set of the set of degree $d$ curves passing thought $k$ points. As we will see in the case of $(4,2)$ and $(3,3)$ it is easy to find degree $2$ curves passing through $4$ points or $3$ points as products of two lines.

One might wonder for any $d$ whether one can find degree $d$ curves passing through $k$ points as products of $d$ of the lines $l_i$. We can frame this in terms of the graph and in  some cases use that to find such curves. A {\bf vertex cover}  $\TV$ of a graph is a subset of the set of vertices  $V$ such that {\em every edge} is adjacent to a element of $\TV$. 

For instance, for the graph $(4,2)$ in Section \ref{Case42} for the square graph $v_1v_2v_3v_4$ the sets  $\{v_1,v_3\}$ and $\{v_2,v_4\}$ are vertex covers but $\{v_1,v_2\}$ is not as the edge $\overline{v_3v_4}$ is not covered. The vertex covers correspond to the required conics - 
$$\{v_i,v_j\} \leftrightarrows l_il_j$$
As the degree increases and as the number of points of tangency increases, we can use this idea to find a convenient generating set of degree $d$ curves passing through a few points.

\section{Results}

Recall that for a graph of type $(k,3d-k)$ the corresponding Humbert Invariant $\Delta$ is 
$2d^2+7-2k$ or $2d^2+8-2k$ depending on whether $d+k$ is odd or even respectively. 

\subsection{Curves of degree 2}

We give an exposition of the three degree 2 cases

\subsubsection{\textbf{Case 1: $(5,1)$ $\Delta=5$}}
	
	This is the original case studied by Humbert. This can be found in Hashimoto-Murabayashi \cite{hamu}, Theorem 2.9. Here the graph is Figure \ref{51}.

\begin{enumerate}
	\item Choose points $q_{12},q_{23},q_{34},q_{45},q_{51}$ and the line $l_6: z=0$. Recall that the coordinates are given in terms of $a_1,a_2$ and $a_3$.

	\item Using simple linear algebra compute  the equation of the conic $C$ passing through the above $5$ points. Denote it by $f\in\C[a_1,a_2,a_3][x,y,z]$. One can also take the determinant of the $6\times 6$ matrix whose first rows are the monomials and all other rows are filled by the coordinates of the $5$ points.
	\item Substitute $z=0$ into $f$. The resulting expression has the form $c_1x^2+c_2y^2+c_3xy$ where $c_1,c_2,c_3\in\Z[a_1,a_2,a_3]$. The zeroes of this are precisely the points of intersection of the line $l_6$ and the conic. 
	\item The line $z=0$ is tangent to the conic if and only if $c_3^2-4c_1c_2=0$ (note that this is the discriminant of a quadratic equation).
\end{enumerate}
The Humbert modular equation is thus $c_3^2-4c_1c_2$. The following is an example code using SAGE for computing this Humbert modular equation.
\begin{verbatim}
	-------------------------------------------------------
	# Define the symbolic variables
	x, y, z, a_1, a_2, a_3 = var('x y z a_1 a_2 a_3')
	
	# Define the first row of the matrix (as a list of expressions)
	first_row = [
	x^2, y^2, z^2, x*y, y*z, x*z
	]
	
	# Define 5 points where these expressions will be evaluated for the subsequent 
	rows points = [(-(a_1+a_2), 2*a_1*a_2, 2), (-(a_3+a_2), 2*a_3*a_2, 2), 
	(-(1+a_3), 2*a_3, 2), (-1, 0, 2), (-a_1, 0, 2)]
	
	# Initialize the matrix with the first row
	matrix = [first_row]
	
	# Add the remaining 9 rows, each being evaluated at the corresponding point
	for point in points:
	x_val, y_val, z_val = point
	row = [expr.subs({x: x_val, y: y_val, z: z_val}) for expr in first_row]
	matrix.append(row)
	
	# Create the matrix
	M = Matrix(matrix)
	
	# Show the resulting matrix
	M
	
	# Equation of the conic curve
	conic = det(M)
	
	
	# Recursively eliminate x and y
	r1 = conic.subs(z==0)
	r2 = r1.subs(y==1)
	
	# Get the coefficients c_1,c_2,c_3
	c2 = r2.subs(x==0)
	c3 = (r2.diff(x)).subs(x==0)
	c1 = (r2.diff(x)).diff(x)/2
	
	# compute Humbert modular equation
	D = c3^2-4*c1*c2
	
	# output the resultant to a text file
	o = open('disc.txt', 'w')
	o.write(D)
	o.close()
	------------------------------------------------------
\end{verbatim}

\subsubsection{\textbf{Case 2: $(4,2)$ $\Delta=4,8$}}

\label{Case42}

To reproduce Theorem 2.12, pg. 286, \cite{hamu}, which is stated without proof, one could proceed as follows. Here the graph is Figure \ref{42}.

\begin{enumerate}
	\item Choose the points $q_{12},q_{23},q_{34},q_{14}$ and lines $l_5:y=0$ and $l_6:z=0$.

	\item Consider the following local picture
	\begin{center}
		\begin{tikzpicture}[x=0.75pt,y=0.75pt,yscale=-1,xscale=1]
			
			\draw   (171,269.68) -- (468.67,269.68) -- (468.67,493.68) -- (171,493.68) -- cycle ;
			\draw    (272,301.68) -- (436.67,439.68) ;
			\draw    (215,319.68) -- (278.67,468.68) ;
			\draw    (190,355.68) -- (356.67,313.68) ;
			\draw    (219,413.68) -- (422.67,396.68) ;
			
			\draw (437,289.08) node [anchor=north west][inner sep=0.75pt]    {$\mathds{P}^2$};
			\draw (282,447.08) node [anchor=north west][inner sep=0.75pt]    {$l_{1}$};
			\draw (409,374.08) node [anchor=north west][inner sep=0.75pt]    {$l_{2}$};
			\draw (279,287.08) node [anchor=north west][inner sep=0.75pt]    {$l_{3}$};
			\draw (192,359.08) node [anchor=north west][inner sep=0.75pt]    {$l_{4}$};
			\draw (375,405.08) node [anchor=north west][inner sep=0.75pt]    {$q_{23}$};
			\draw (288,334.08) node [anchor=north west][inner sep=0.75pt]    {$q_{34}$};
			\draw (225,330.08) node [anchor=north west][inner sep=0.75pt]    {$q_{14}$};
			\draw (235,415.08) node [anchor=north west][inner sep=0.75pt]    {$q_{12}$};

		\end{tikzpicture}
		\vspace{0.1cm}
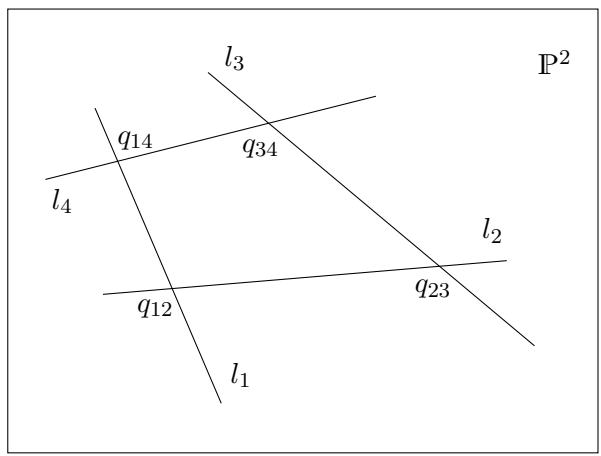
		\text{The $4$ points chosen as intersection of the lines}
		\label{fig2}
	\end{center}
	\item The equation \[f = l_1l_3t_0+l_2l_4t_1\in\C[a_1,a_2,a_3,a_4][x,y,z][t_0,t_1]\] defines a pencil of conics passing through the $4$ chosen points exactly once.
	
	\item Substitute $y=0$ into $f$. The resulting expression looks like 
	$$c_1(t_0,t_1)x^2+c_2(t_0,t_1)z^2+c_3(t_0,t_1)xz.$$
	where $c_i(t_0,t_1)$ are linear in $t_0$ and $t_1$. 
	
	Similar to case 1 above, the line $l_5:y=0$ is tangent to the family of conics if and only if $D_1 = c_3^2-4c_1c_2\in\Z[a_1,a_2,a_3,a_4][t_0,t_1]$ vanishes. Note that $D_1$ is homogeneous in $t_0,t_1$ of degree $2$.
	
	\item Similarly we require the line  $l_6:z=0$ to be a tangent. Substitute $z=0$ into $f$. The resulting expression looks like $d_1(t_0,t_1)x^2+d_2(t_0,t_1)y^2+d_3(t_0,t_1)xy$. Similar to case 1 above, the line $l_6$ is tangent to the family of conics if and only if $D_2 = d_3^2-4d_1d_2\in\Z[a_1,a_2,a_3,a_4][t_0,t_1]$ vanishes. Note that $D_2$ is homogeneous in $t_0,t_1$ of degree $2$.

	\item Both lines $l_5$ and $l_6$ are tangent to $f$ if and only if $D_1$ and $D_2$ have a common root in $t_0,t_1$. Using eq.\hyperref[resultantoftwopoly]{(3)}, let $r=\res_{t_1}(D_1,D_2)\in\Z[a_1,a_2,a_3,a_4][t_0]$. This eliminates the variable $t_1$.

	\item Substitute $t_0=1$ into $r$.
\end{enumerate}
The resulting expression obtained above is Humbert modular equation corresponding to $\Delta=8$ or $4$ as stated in \cite{hamu}, Th. 2.12 on pg. 286.
\begin{verbatim}
	-------------------------------------------------------------
	# Define the symbolic variables
	x, y, z, a_1, a_2, a_3, a_4, t_0, t_1 = var('x y z a_1 a_2 a_3 a_4 t_0 t_1')
	
	# Define the first row of the matrix (as a list of expressions)
	first_row = [
	x^2, y^2, z^2, x*y, y*z, x*z
	]
	
	# Define the lines
	l1 = y+2*a_1*x+a_1^2*z
	l2 = y+2*a_2*x+a_2^2*z
	l3 = y+2*a_3*x+a_3^2*z
	l4 = y+2*a_4*x+a_4^2*z
	
	# Define the pencil of conics
	conic = l1*l3*t_0+l2*l4*t_1
	
	# Substitute y=0 and z=0
	conicy = conic.subs(y==0)
	conicz = conic.subs(z==0)
	
	# compute the discriminants similar to the previous case
	to get D1 and D2. 
	# Eliminate variable t_1 from D1 and D2
	r = D1.resultant(D2,t_1)
	
	# Substitute t_0=1
	r = r.subs(t_0==1)
	
	# output the resultant to a text file
	o = open('disc.txt', 'w')
	o.write(r)
	o.close()
	------------------------------------------------------------
\end{verbatim}

\subsubsection{\textbf{Case 3: $(3,3)$ $\Delta=9$}}

Here the graph is Figure \ref{33}. Along the same lines one can consider $f = t_0l_1l_2+t_1l_2l_3+t_2l_1l_3$ and proceed exactly as in Case 2 above. The modular equation corresponds to $\Delta=9$.

\subsubsection{\textbf{Case 4: (0,6)}}

\begin{center}
	\begin{tikzpicture}[scale=0.5]
		\node[circle, draw, fill=blue!20] (v1) at (0,0) {1};
		\node[circle, draw, fill=blue!20] (v2) at (2,0) {2};
		\node[circle, draw, fill=blue!20] (v3) at (3,2) {3};
		\node[circle, draw, fill=blue!20] (v4) at (2,4) {4};
		\node[circle, draw, fill=blue!20] (v5) at (0,4) {5};
		\node[circle, draw, fill=blue!20] (v6) at (-1,2) {6};
		

		\draw (v6) to [out=135, in=225, loop] (v6);
		\draw (v5) to [out=135, in=225, loop] (v5);
		\draw (v4) to [out=45, in=-45, loop] (v4);
		\draw (v1) to [out=135, in=225, loop] (v6);
		\draw (v2) to [out=45, in=-45, loop] (v5);
		\draw (v3) to [out=45, in=-45, loop] (v4);

	\end{tikzpicture}
	
	The case (0,6). 
	
\end{center}
This is another apparently admissible case - though we require $k \geq 3$ to apply Theorem \ref{birkwilh} - when there is a conic tangent to all the six lines $l_i$. 
However, in the Kummer plane corresponding to an Abelian surface it is always the case that there is a conic tangent to the six lines. One can consider the more general situation of the $K3$ surfaces determined by the minimal desingularization of the double cover of $\CP^2$ ramified at six lines. This is a four dimensional moduli space and in that the moduli of Kummer surfaces of Abelian surfaces is precisely the three dimensional moduli where the six lines are tangent to a conic. 

\subsection{Curves of degree $3$}

We consider a general degree $3$ planar curve $C$ defined by the following equation
\begin{align*}
f=\sum_{i+jk=3} c_{ijk}x^iy^jz^k
\end{align*}

Recall that our strategy here is to find a possibly smooth cubic curve which meets the singular sextic $S=\bigcup_{i=1}^6 l_i$ only at double points - either the points $q_{ij}=l_i \cap l_j$ or a point $t_i$ where $l_i$ meets $C$ tangentially. 

It has to meet each line $l_i$ at three points - and the only possible combinations are three points of the form $q_{ij}$ or one point $q_{ij}$ and a point $t_i$. Since there are six lines there can be at most six points of tangency of $S$ and $C$.  

A homogeneous cubic equation is determined by $10$ coefficients so the the space of cubic curves in $\CP^2$ is parameterized by $\CP^9$. The condition that it passes through a point imposes a linear relation on the coefficients and hence the dimension drops by $1$. For $k$ points in general position the dimension will drop by $k$. So the set of cubics passing through $k$ points is parameterized by $\CP^{9-k}$

Our strategy is to consider the space spanned by the $9-k$ cubics passing though $k$ points and find those tangent to a $9-k$ lines. 

Let $C_0,\dots,C_{9-k}$ be the cubic curves passing through $k$ points of the form $q_{ij}$ 
defined by cubic polynomials  $f_i$. Since $q_{ij}$ are determined by $a_1,a_2,a_3$ of the original hyperelliptic curve, the coefficients of the  $f_i$ are determined by them as well. Note that if $\{f_i\},0 \leq r$ pass through a point $P$ so does any $f$ of the form $\sum_0^r a_i f_i$.  

We have the following theorem

\begin{thm} Let $f=\sum_{i=0}^{9-k} t_i f_i$. Let $g_i=f \circ l_i$, which is a homogeneous polynomial in $2$ variables and coefficients determined by $t_i$ and $a_1,a_2,a_a$. Then $\disc(g_i)$ is a polynomial function of $t_i$ and $a_1,a_2,a_3$. Let $D(f)=\disc(f)=\res_{2,2,2}(f_x,f_y,f_z)$. This too is a polynomial in function of $t_i$ and $a_1,a_2,a_3$.

For each $k \geq 3$ the Kummer plane contains  a rational cubic curve passing through $k$ of the points $q_{ij}$ and meeting $9-k$ of the  singular lines $l_i$ tangentially if and only if 
$$\res(\disc(g_0),\dots,\disc(g_{9-k}),\disc(f))=0$$
where here we are thinking of the $g_i$ and $f$  as functions of $t_0,\dots,t_{9-r}$, hence 
the last expression is a polynomial function in $a_1,a_2$ and $a_3$.  
\label{cubiccase}

\end{thm}

\begin{proof}
Consider the $9-k$ dimensional projective space spanned by these, namely, cubic polynomials  of the form 
$$f=\sum_{i=0}^{9-k} t_i f_i$$
We want to find a cubic curve passing through the $k$ points and tangent to $9-k$ lines $l_i$. Let $l_i$ be one of those lines.  Let $g_i=f\circ l_i$. Then from Proposition \ref{sylvester} we have 
$$l_i \text{ is tangent to the curve } C=V(f) \Leftrightarrow \disc(g_i)=\res_{2,2}(g_{i,x},g_{i,y})=0$$
where $x$ and $y$ are co-ordinates such that  $g_i$ is a homogeneous polynomial of degree $2$.

Note that $\disc(g_i)$ is a polynomial in the $t_0,\dots,t_{9-k}$ with coefficients polynomials in  $a_1,a_2$ and $a_3$. Hence if $\res(\disc(g_1),\dots,\disc(g_{9-k}))=0$ for some $t_1,\dots,t_{9-k}$, $C$ will be tangent to all the lines $l_i$. Since these are $9-k$ conditions, such a $C$ always exists. 

However, this may not be rational. $C$ is rational if and only if $\disc(f)=0$. This determines polynomial conditions  on $a_1,a_2$ and $a_3$.

\end{proof}

Putting all this together, we have  the following theorem 

\begin{thm} Let 
	$$C: y^2=x(x-1)(x-{a_1})(x-{a_2})(x-{a_3})$$
	be a genus two hyperelliptic curve in $\CP^2$ where $a_i \in \CP^1-\{0,1,\infty\}$. Let $A=J(C)$. 
	
	For $3 \leq k \leq 12$, let 
	$$F_{3,k}(a_1,a_2,a_3)=\res(\disc(g_0),\dots,\disc(g_{9-k}),\disc(f))$$
	Then if  $F_{3,k}(a_1,a_2,a_3)=0$, $A$ has real multiplication by $\Delta'$ where  $0< \Delta' \leq 25-2k$ if $k$ is even and $0<\Delta' \leq 26-2k$ if $k$ is odd.  

\end{thm}

The discriminant of an element $f=\sum_{i=0}^{9-k} t_if_i$  of the $9-k$  dimensional family of cubics is homogeneous of degree $12$ in $t_0,\dots,t_{9-k}$. 

To compute the degree of $\disc(g_i)$ note that in Proposition \ref{sylvester}, the entries of the matrix are linear in the $t_i$ and the degree is $3$ so the matrix is of rank $2 \times 3-2=4$. 
Hence each $\disc(g_i)$ is homogeneous of degree $4$ 

Thus the resultant $\res(D(f),\disc(g_i),\dots,\disc(g_{(9-k)}))$ is a resultant of mixed degree polynomials.

\subsubsection{\textbf{Case 1:} $(9,0)$ $\Delta=9$}
As remarked before, the bi-partite case does not lead to anything new as such cubics always exist. Here we have to consider the non-bipartite graph $(9,0)b$

\label{Case90b}

\begin{center}
	\begin{tikzpicture}[scale=0.5]
		\node[circle, draw, fill=blue!20] (v1) at (0,0) {1};
		\node[circle, draw, fill=blue!20] (v2) at (2,0) {2};
		\node[circle, draw, fill=blue!20] (v3) at (3,2) {3};
		\node[circle, draw, fill=blue!20] (v4) at (2,4) {4};
		\node[circle, draw, fill=blue!20] (v5) at (0,4) {5};
		\node[circle, draw, fill=blue!20] (v6) at (-1,2) {6};
		
		\draw (v1) -- (v2);
		\draw (v1) -- (v3);
		\draw (v1) -- (v6);
		\draw (v2) -- (v5);
		\draw (v2) -- (v3);
		\draw (v3) -- (v4);
		\draw (v5) -- (v4);
		\draw (v5) -- (v6);
		\draw (v6) -- (v4);

	\end{tikzpicture}
	
The graph $(9,0)b$.
\end{center}

We proceed as follows.

\begin{enumerate}
	
\item The nine points are $q_{12}, q_{13}, q_{16}, q_{23}, q_{25}, q_{34}, q_{45}, q_{46}$ and $q_{56}$
    
\item Compute the defining equation: $f\in \C[x,y,z]$ as the determinant of a $10\times 10$ matrix using the $9$ points above.
    
\item Computing the discriminant: eliminate one variable at a time using Proposition \ref{resultantoftwopoly}. Denote by $\res_s$ the resultant of two polynomials with respect to any variable $s$.  
    
\item Compute $r_1 = \res_x(f_x,f_y)\in\C[y,z]$ and $r_2 = \res_x(f_y,f_z)\in\C[y,z]$. This eliminates the variable $x$.
    
\item Compute $r_3 = \res_y(r_1,r_2)\in\C[z]$. This eliminates the variable $y$.
    
\item Substitute $z=1$ into $r_3$, this is the discriminant of $f$, $D(f)\in \Z[a_1,a_2,a_3,a_4]$.
\end{enumerate}
The discriminant can be also computed using Proposition \ref{sylvester}. This would be the irreducible resultant while the steps above might produce a multiple of the resultant. In any case, computing the discriminant is quite a CPU-intensive task. It is easier if two of more of the $a_i$'s are given numerical values. 

The polynomial $D(f)$ is the Humbert modular equation corresponding to $\Delta = 8$ (cf. \hyperref[valdelta]{table}).

The computation can be done using any computer algebra system. Here is an example using SAGE.
\label{case1}
\begin{verbatim}
------------------------------------------------------------------
# Define the symbolic variables
x, y, z, a_1, a_2, a_3, a_4 = var('x y z a_1 a_2 a_3 a_4')

# Define the first row of the matrix (as a list of expressions)
first_row = [
    x^3, y^3, z^3, x^2*y, x^2*z, y^2*x, y^2*z, z^2*x, z^2*y, x*y*z
]

# Define 9 points where these expressions will be evaluated for the subsequent 
rows points = [(-(a_1+a_4), 2*a_1*a_4, 2), (-a_1, 0, 2), (-1, 2*a_1, 0), 
(-(a_2+a_3), 2*a_2*a_3, 2), (-(a_2+a_4), 2*a_2*a_4, 2), (-a_2, 0, 2), 
(-(a_3+a_4), 2*a_3*a_4, 2), (-a_3, 0, 2), (1, 0, 0)]

# Initialize the matrix with the first row
matrix = [first_row]

# Add the remaining 9 rows, each being evaluated at the corresponding point
for point in points:
    x_val, y_val, z_val = point
    row = [expr.subs({x: x_val, y: y_val, z: z_val}) for expr in first_row]
    matrix.append(row)

# Create the matrix
M = Matrix(matrix)

# Show the resulting matrix
M

# Equation of the cubic curve
cubic = det(M)

# Get the partial derivatives
cubicx = cubic.diff(x)/128
cubicy = cubic.diff(y)/64
cubicz = cubic.diff(z)/256

# Recursively eliminate x and y
r1 = cubicx.resultant(cubicy,x)
r2 = cubicy.resultant(cubicz,x)
r3 = r1.resultant(r2,y)

# output the resultant to a text file
o = open('disc.txt', 'w')
o.write(r3.subs(z==1))
o.close()
----------------------------------------------------------------
\end{verbatim}

\subsubsection{\textbf{Case 2:} $(8,1)$ $\Delta=8$}
\label{Case81}

\begin{center}
	\begin{tikzpicture}[scale=0.5]
		\node[circle, draw, fill=blue!20] (v1) at (0,0) {1};
		\node[circle, draw, fill=blue!20] (v2) at (2,0) {2};
		\node[circle, draw, fill=blue!20] (v3) at (3,2) {3};
		\node[circle, draw, fill=blue!20] (v4) at (2,4) {4};
		\node[circle, draw, fill=blue!20] (v5) at (0,4) {5};
		\node[circle, draw, fill=blue!20] (v6) at (-1,2) {6};
		
		\draw (v1) -- (v2);
		\draw (v1) -- (v3);
		\draw (v1) -- (v5);
		\draw (v2) -- (v3);
		\draw (v2) -- (v4);
		\draw (v3) -- (v4);
		\draw (v4) -- (v5);
		\draw (v5) -- (v6);
			
		\draw (v6) to [out=135, in=225, loop] (v6);
	\end{tikzpicture}
	
	The graph $(8,1)$.
\end{center}

The following steps may be taken.
\begin{enumerate}
    \item Pick $8$ double points corresponding to the graph above. There are eight double points 
    \[\{q_{12}, q_{13}, q_{15}, q_{23}, q_{24}, q_{34}, q_{45}, q_{56}\}\] 
    along with a tangent to the line $l_6$ at $t_6$. This corresponds to the graph depicted above.

    \item Let $p_1$ and $p_2$ be two points in $\CP^2$ at least one of which is not on $S$.

    \item Let $f_1$ be the equation of a cubic curve passing through the eight points as well as $p_1$. Similarly let $f_2$ be the equation of the cubic curve through $p_2$ and the eight double points. 
    
    \item Consider the family $f(x,y,z)=t_1f_1(x,y,z)+t_2f_2(x,y,z)$ over $\CP^1$. 
    
    \item Compute the discriminant $D(f)$ using eq.\hyperref[eq6]{(6)} or the recursive method used in the previous case.
    
    \item For the singular line $l_6:z=0$, using eq.\hyperref[sylvester]{(5)} compute $\disc(g_6)$ where $g_6=f\circ l_6$.

    \item Compute the resultant $\res(D(f),\disc(g_6))$ (using proposition \hyperref[resultantoftwopoly]{\ref{resultantoftwopoly}}) after dehomogenizing to an affine variable, say $t = t_1/t_2$. We know $D(f)$ and $g_6$ are homogeneous of degrees $12$ and $4$ respectively in $t_1,t_2$ (cf. proposition \hyperref[resultantdeg]{\ref{resultantdeg}}).

\end{enumerate}

Note that it is possible to compute the resultant in step 7 above first by fixing one of the projective coordinates and then substituting the value $1$ into the resulting expressing. That is, dehomogenizing can be done at the end. We shall do this in all the successive cases.
This case corresponds to $\Delta=9$. 

Along the lines of Section \ref{lines} one can see that the set $\{v_5,v_2,v_3\}$ is a vertex cover so one of the cubics passing through the $8$ points is $l_5l_2l_3$. However there do not appear to be any more of this type.

We also get the following corollary (which is well known)
\begin{cor}
    There are exactly $12$ rational cubic curves passing through $8$ points in general position in $\CP^2(\C)$.
\end{cor}
\begin{proof}
    Given $8$ points, consider a $1$-parameter family of cubics over, $f_t=f_1+tf_2$ where $f_1$ and $f_2$ are the equations of  cubic curves (possibly reducible) such that $f_1$ is not a scalar multiple of $f_2$.
    
    The rational curves in the family $f_t$ correspond to cubic equations with vanishing discriminant. Using eq.\eqref{eq6}, we can compute $D(f_t)$.  From the formula for the discriminant, $D(f_t)$ is a degree $12$ polynomial in $\C[t]$. Hence $D(f_t)=0$ at precisely $12$ values of $t$. 
\end{proof}

\subsubsection{\textbf{Case 3:} $(7,2)$ $\Delta=12$}

Here there are two cases, the graphs $(7,2)a$ and $(7,2)b$ in Figure \ref{72}. One proceeds along the same lines but it is interesting to note that here one can find convenient vertex covers which give a generating set of the degree $3$ polynomials passing through $7$ points. 

For instance, in the case of $(7,2)a$, the seven points are $q_{12}, q_{13}, q_{16}, q_{23}, q_{24}, q_{34}$ and $q_{45}$.  The space of cubics passing through these seven points is parametrized by $\CP^2$ and so one needs $3$ cubics. 

One can see that $\{v_1,v_4,v_2\}$, $\{v_1,v_4,v_3\}$ and $\{v_6,v_5,v_3\}$ are three vertex covers and so  $f_{142}=l_1l_4l_2$, $f_{143}=l_1l_4l_3$ and $f_{653}=l_6l_5l_3$ are three cubics passing through the seven points. All such cubics are given by 
$$t_0f_{142}+t_1f_{143}+t_2f_{653}$$

\subsubsection{\textbf{Case 4:} $(6,3)$ $\Delta=13$}

\begin{center}
	\begin{tikzpicture}[scale=.5]
		\node[circle, draw, fill=blue!20] (v1) at (0,0) {1};
		\node[circle, draw, fill=blue!20] (v2) at (2,0) {2};
		\node[circle, draw, fill=blue!20] (v3) at (3,2) {3};
		\node[circle, draw, fill=blue!20] (v4) at (2,4) {4};
		\node[circle, draw, fill=blue!20] (v5) at (0,4) {5};
		\node[circle, draw, fill=blue!20] (v6) at (-1,2) {6};
		
		\draw (v1) -- (v2);
		\draw (v2) -- (v3);
		\draw (v3) -- (v1);
		\draw (v3) -- (v4);
		\draw (v2) -- (v5);
		\draw (v1) -- (v6);

		\draw (v4) to [out=45, in=-45, loop] (v4);
		\draw (v5) to [out=135, in=225, loop] (v5);
		\draw (v6) to [out=135, in=225, loop] (v6);
		
	\end{tikzpicture}
	
	The graph $(6,3)$.
\end{center}

\label{Case63}

To obtain the Humbert modular equation we do the following. 

\begin{enumerate}
    \item Pick $6$ double points, for example, corresponding to the following graph

    \item Start with the parametric equation of the cubic $f\in\C[x,y,z][t_0,t_1,t_2,t_3]$ \[f=t_0C_0+t_1C_1+t_2C_2+t_3C_3\]
    
    \item Compute the discriminant $D(f)$ using eq.\hyperref[eq6]{(6)}.
    
    \item Compute $T_1(f),T_2(f),T_3(f)\in\C[t_0,t_1,t_2,t_3]$ for each tangent using eq.\hyperref[sylvester]{(5)}.
    
    \item Now recursively eliminate $t_1,t_2,t_3$ using eq.\hyperref[resultantoftwopoly]{(3)}.
    
    \item Eliminate $t_1$: $r_1 = \res_{t_1}(D(f),T_1(f))$, $r_2 = \res_{t_1}(T_1(f),T_2(f))$, $r_3 = \res_{t_1}(T_2(f),T_3(f))$.
    
    \item Eliminate $t_2$: $s_1 = \res_{t_2}(r_1,r_2)$, $s_2 = \res_{t_2}(r_2,r_3)\in\Z[a_1,a_2,a_3,a_4][t_0,t_3]$.
    
    \item Eliminate $t_3$: $u = \res_{t_3}(s_1,s_2)\in\Z[a_1,a_2,a_3,a_4][t_0]$.
    
    \item Finally substitute $t_0=1$ into $u$.
\end{enumerate}
This case corresponds to $\Delta = 13$.

\subsection{Remarks on higher degree curves}

In this section we outline the procedure for higher degree curves. The procedure is similar- only computationally harder.

The space  of curves of degree $d$ in $\CP^2$, $S_d$  is $\frac{(d+1)(d+2)}{2}-1$ dimensional. Hence the condition that the curve passes through $3d-1$ double points  does not determine a unique curve. In fact there is a  $\frac{(d+1)(d+2)}{2}-1-(3d-1) = \frac{(d-1)(d-2)}{2}=g$ family of degree $d$ curves through $3d-1$ points. Let $T_d$ be this family. Note that the dimension is the geometric genus of a generic fibre of this family - namely a smooth curve of degree $d$.    

Suppose $F_{\ut}=F_{t_0,\dots,t_{g}}$ is the equation of such a curve $C_{\underline{t}}$. The condition that it has a singular point is given by $\disc(F_{\ut})=0$. The set of such equations  determines a codimensional $1$ subvariety $T_1$ of the space of degree $d$ curves passing through the $3d-1$ double points 
	
Over $T_1$, the set of singular points is 
$$\rm{Sing}(C_{T_1})=V((F_{\ut})_x,(F_{\ut})_y,(F_{\ut})_z).$$
namely the set of common zeroes of all the partial derivatives.  Over $T_1$, there is at least one section  $\sigma$ of the $\CP^2$ bundle over $T_1$ such that $\sigma(\ut)$ is a singular point of $C_{\ut}$.

We can then blow up the point $\sigma_T$  in $\CP^2_T$ and consider the strict transform of $\tilde{C}_{T_1}$ of $C_{T_1}$.  This is a new family over $T_1$ with generic geometric genus one less than that of $C_T$. The subvariety of such curves with $\disc(\tilde{F}_{\ut})=0$ -- where $\tilde{F}_{\ut}$ is the equation of $\tilde{C}_{\ut}$ -- is again of codimension one. Repeating this process $g$ times results in a $0$ dimensional family of degree $d$ curves of genus $0$ - that is, rational curves,   passing through $3d-1$ double points. 

Recall that we considered the resultant 
$$\res (\disc(g_0),\dots,\disc(g_{3d-k}))$$
which is a function of $t_0,\dots,t_{3d-k}$ and the $a_1,a_2,a_3$.  This parametrized the space of degree $d$ curves passing through $k$ points and tangent to $3d-k$ lines. Considering 
$$F_{d,k}(a_1,a_2,a_3)=\res{(\disc(F),\disc(\tilde{F}),\dots,\disc(g_0),\dots,\disc(g_{3d-k}))}$$
results in a function of $a_1,a_2$ and $a_3$ and is the modular equation corresponding to $\Delta(d,k)$. 
	
\subsection{Remarks on higher genus curves}

We have considered only genus $0$ (rational) curves, However, this does not cover all possible $\Delta$. Birkenhake and Wilhelm \cite{biwi}, Theorem 7.6 show that one sometimes needs to use higher genus curves. The theorem asserts that if the Kummer plane admits a $g$-dimensional linear system passing through $3d$ double points then the corresponding Abelian surface has real multiplication by some $\Delta$ determined by $d,k$ and $g$, where $k$ is the number of ponts $q_{ij}$ that the linear system passes through. 

This can also be treated in the same way.  For similar reasons as above in general, for the Kummer plane of any Abelian surface, one will have a $g$ dimensional linear system of genus $g$ and degree $d$ curves passing through $3d-1$ double points. These curves will meet the sextic at two other points and when they coincide it means that the corresponding Abelian surface has real multiplication by some $\Delta>0$. As above this can be translated in to a condition on $a_1, a_2$ and $a_3$.

\bibliographystyle{alpha}
\bibliography{references.bib}

\begin{tabular}[t]{l@{\extracolsep{8em}}l} 
	Rahul Mistry & Ramesh Sreekantan\\
	Statistics and Mathematics Department & Statistics and Mathematics Department\\
	Indian Statistical Institute & Indian Statistical Institute \\
	8th Mile, Mysore Road & 8th Mile, Mysore Road \\
	Jnanabharathi & Jnanabharathi \\
	Bengaluru 560 059 & Bengaluru 560 059 \\ 
	rs\_math2003@isibang.ac.in & rameshsreekantan@gmail.com 
\end{tabular}

\end{document}